\documentclass{amsart}

\usepackage{amssymb,mathrsfs}
\usepackage{graphicx}
\usepackage{algorithmic}
\usepackage{subfigure}

\usepackage{epsfig}

\usepackage{enumerate}        

\graphicspath{{pics/}}

\newtheorem{theorem}{Theorem}[section]
\newtheorem{lemma}[theorem]{Lemma}

\newtheorem{proposition}[theorem]{Proposition}
\newtheorem{corollary}[theorem]{Corollary}

\theoremstyle{definition}
\newtheorem{definition}[theorem]{Definition}
\newtheorem{notation}[theorem]{Notation}

\newcommand{\wR}[1]{w_R(#1)}
\newcommand{\wL}[1]{w_L(#1)}
\newcommand{\wC}[1]{w_C(#1)}

\newcommand {\ba}{\mathbf{a}}
\newcommand {\bx}{\mathbf{x}}
\newcommand {\by}{\mathbf{y}}
\newcommand {\bz}{\mathbf{z}}
\newcommand {\bu}{\mathbf{u}}
\newcommand {\x}{\mathbf{x}}
\newcommand {\X}{\mathbf{X}}
\newcommand {\R}{\mathbb{R}}
\newcommand {\Z}{\mathbb{Z}}
\newcommand {\F}{F_n}
\newcommand{\ol}[1]{\overline{#1}}
\newcommand{\la}{\langle}
\newcommand{\ra}{\rangle}

\newcommand{\FSC}[1][n]{\mathcal{FS}_{#1}}
\newcommand{\ESC}[1][n]{\mathcal{ES}_{#1}}

\newcommand{\FS}[1][n]{\mathcal{FS}_{#1}^1}
\newcommand{\ES}[1][n]{\mathcal{ES}_{#1}^1}

\newcommand{\FFZ}[1][n]{\mathcal{FF}_{#1}^1}
\newcommand{\FF}[1][n]{\mathcal{F}_{#1}^1}
\newcommand{\FFZC}[1][n]{\mathcal{FF}_{#1}}
\newcommand{\FFC}[1][n]{\mathcal{F}_{#1}}

\newcommand{\Aut}[1][F_n]{\mathop{\mathrm{Aut}}(#1)}
\newcommand{\Inn}[1][F_n]{\mathop{\mathrm{Inn}}(#1)}
\newcommand{\Out}[1][F_n]{\mathop{\mathrm{Out}}(#1)}

\newtheorem*{theoremquasiflat}{Theorem \ref{thm:quasiflat}}
\newtheorem*{theoremlowerbound}{Theorem \ref{thm:lowerbound}}
\newtheorem*{corollarynothyperbolic}{Corollary \ref{cor:nothyperbolic}}
\newtheorem*{corollaryinfinitedimension}{Corollary \ref{cor:infinitedimension}}
\newtheorem*{corollarynotqi}{Corollary \ref{cor:notqi}}
\newtheorem*{corollarynotqi2}{Corollary \ref{cor:notqi2}}

\begin{document}

\title{On the geometry of a proposed curve complex analogue for $\Out$}
\author[L.Sabalka]{Lucas~Sabalka}
      \address{Department of Mathematical Sciences\\
               Binghamton University\\
               Binghamton NY  13902-6000
}
\author[D.Savchuk]{Dmytro~Savchuk}

\begin{abstract}

The group $\Out$ of outer automorphisms of the free group has been an object of active study for many years, yet its geometry is not well understood.  Recently, effort has been focused on finding a hyperbolic complex on which $\Out$ acts, in analogy with the curve complex for the mapping class group.  Here, we focus on one of these proposed analogues:  the edge splitting complex $\ESC$, equivalently known as the separating sphere complex.  We characterize geodesic paths in its 1-skeleton $\ES$ algebraically, and use our characterization to find lower bounds on distances between points in this graph.

Our distance calculations allow us to find quasiflats of arbitrary dimension in $\ESC$.  This shows that $\ESC$: is not hyperbolic, has infinite asymptotic dimension, and is such that every asymptotic cone is infinite dimensional.  These quasiflats contain an unbounded orbit of a reducible element of $\Out$.  As a consequence, there is no coarsely $\Out$-equivariant quasiisometry between $\ESC$ and other proposed curve complex analogues, including the regular free splitting complex $\FSC$, the (nontrivial intersection) free factorization complex $\FFZC$, and the free factor complex $\FFC$, leaving hope that some of these complexes are hyperbolic.

\end{abstract}

\maketitle

\section{Introduction}\label{sec:intro}

Let $\Out$ denote the group of outer automorphisms of the free group
$\F$ of rank $n$, where we assume throughout this paper that $n >
2$.  We wish to study the geometry of $\Out$, by examining the
geometry of certain spaces on which group acts.  There is a strong
analogy between $\Out$ and the mapping class group of a surface on
the one hand and arithmetic groups on the other, which has been
pursued quite fruitfully in the last couple of decades.  This
approach began in earnest with the foundational paper of Culler and
Vogtmann \cite{culler_v:outer_space}, which introduced Outer Space,
the analogue for $\Out$ of Teichm\"uller space for the mapping class
group and of symmetric spaces for arithmetic groups.  The work that
followed has yielded numerous statements about the topological,
homological, and cohomological properties of $\Out$ and the spaces
it acts upon -- see for instance
\cite{vogtmann:aut_fn_and_outer_space} for an excellent survey.

While the topology of Outer Space is well-understood, its geometry
is not.  In contrast, the geometries of Teichm\"uller space and the
symmetric spaces are well-studied.  One key ingredient for the study
of Teichm\"uller space is the celebrated result of Masur and Minsky,
who proved that the curve complex is hyperbolic \cite{MasurMinsky}.
The \emph{curve complex} is the complex whose vertex set is the set
of isotopy classes of simple closed curves on the surface, and where
a $k$-simplex corresponds to $k+1$ isotopy classes which have
representatives that are disjoint.  Moreover, there is a `nice' map
from Teichm\"uller space to the curve complex, so that the
hyperbolicity of the curve complex has led to many further
statements on the geometry of Teichm\"uller space and the mapping
class group \cite{BehrstockKleinerMinskyMosher}.  The curve complex
has been used, for instance, to prove quasiisometric rigidity of the
mapping class group.  The analogous key ingredients in the study of
arithmetic groups are Tits buildings, which again yield, for
instance, rigidity theorems.  The `correct' analogue for $\Out$ is
still unknown, and much recent effort has been directed towards
finding one -- in particular, one which is hyperbolic.

There are many possible ways of defining such an analogue.  We will
formally define the most relevant two soon, but we leave definitions
of the remaining complexes and graphs to the references. Before we list some of proposed analogues, let us mention that in most cases they are defined as complexes, but for our purposes (detecting hyperbolicity and distinguishing the spaces up to quasiisometry) it is enough to consider just 1-skeletons of the complexes. For each complex we will denote its 1-skeleton by adding superscript `1' to the notation of the complex. Although we will rigorously define and work only with 1-skeletons of the complexes to simplify exposition, our results apply to the corresponding complexes as well.

Complexes and graphs which deserve mention as possible analogues include:  the \emph{sphere complex} \cite{Hatcher:HomologicalStability}, also called the \emph{free splitting complex} $\FSC$, and its $1$-skeleton $\FS$, called the \emph{free splitting graph} \cite{AramayonaSouto}; the \emph{(common refinement) free factorization complex}, defined in \cite{HatcherVogtmann:CerfTheory} for $\Aut$, whose $\Out$ version we call the \emph{edge splitting complex} $\ESC$ in this paper; the \emph{free factor complex} $\FFC$ (also defined initially for $\Aut$ in \cite{HatcherVogtmann:Complex_of_free_factors}); and the \emph{intersection graph} of Kapovich and Lustig \cite{kapovich_l:analogues_of_curve_complex}.  Kapovich and Lustig \cite{kapovich_l:analogues_of_curve_complex} in fact list 9 graphs which could be an analogue of the curve complex.  They include the $1$-skeleton of the edge splitting complex which we call the \emph{edge splitting graph} $\ES$ (called the \emph{free splitting graph} in~\cite{kapovich_l:analogues_of_curve_complex}, though they do not allow HNN-extensions as vertices) and the $1$-skeleton of the free factor complex which we call the \emph{free factor graph} $\FF$ (called the \emph{dominance graph} in~\cite{kapovich_l:analogues_of_curve_complex}).

Kapovich and Lustig claim that, among the 9 graphs they list, there are at most 3 quasiisometry classes.  Representatives of the three mentioned quasiisometry classes are the edge splitting graph, the free factor graph, and the intersection graph.  We intend to show that the class containing the edge splitting graph cannot be coarsely $\Out$-equivariantly quasiisometric to the other two.  For our purposes, it will be more convenient to use what we call the \emph{(nontrivial intersection) free factorization graph} $\FFZ$ instead of the free factor graph as a representative of the second quasiisometry class.  Note that our free factorization graph is not the $1$-skeleton of Hatcher and Vogtmann's (common refinement) free factorization complex (herein called the edge splitting graph), and that this graph was called the \emph{dual free splitting graph} in~\cite{kapovich_l:analogues_of_curve_complex}, though again in the latter reference they did not allow HNN-extensions as vertices.  We now define the edge splitting graph and the free factorization graph.

\begin{definition}[$\ES$ and $\FFZ$]
For $n > 2$, define the \emph{edge splitting graph}, denoted
$\ES$, to be the graph whose vertices correspond to conjugacy
classes  $[\la x_1, \dots, x_k \ra * \la x_{k+1}, \dots, x_n\ra]$ of
free factorizations $\la x_1,\dots, x_k\ra * \la x_{k+1}, \dots,
x_n\ra$ of $\F$ into two nontrivial free factors.  Two vertices of
$\ES$ are connected with an edge if there exists a free
factorization in each conjugacy class such that the two
factorizations have a common refinement which is a free
factorization into three nontrivial factors.

The \emph{(nontrivial intersection) free factorization graph} $\FFZ$
has the same vertex set as $\ES$.  Two vertices $[A*B]$ and
$[C*D]$ are connected with an edge in $\FFZ$ if one of $A \cap C$, $A \cap D$, $B \cap C$, or $B \cap D$ is nontrivial.
\end{definition}

The name of the edge splitting graph comes from Bass-Serre theory,
where such a free factorization is a graph of groups decomposition
of $\F$ with underlying graph having exactly two vertices and a
single edge (with trivial edge group) between them.  Note that the related \emph{free splitting graph} $\FS$ (the
$1$-skeleton of the free splitting complex or equivalently the
sphere complex) is defined similarly to $\ES$, but also allows conjugacy classes of splittings of $F_n$ as HNN-extensions as vertices.

There are alternate ways to define each of these objects.  In particular, the edge splitting graph $\ES$ is also known as the \emph{separating sphere graph}, whose vertices are homotopy classes of separating essential embedded spheres in a 3-manifold with fundamental group $F_n$, and two vertices are adjacent if they have disjoint representatives.  The free factorization graph can equivalently be defined in terms of Bass-Serre theory, where vertices are Bass-Serre trees of free splittings up to $\Out$-equivariant isometry, and adjacency corresponds to having a common elliptic element.

Note there is a natural action of $\Out$ on all of these spaces, where for $\ES$ and $\FFZ$ the action is induced by the action of $\Out$ on free factorizations.

Not all that much is known about either of these graphs or their siblings.  Hatcher showed that the sphere complex, which contains Outer Space as a dense subspace, is contractible (this gives an alternate proof of contractibility of Outer Space
\cite{culler_v:outer_space}, as the contraction restricts to a contraction of Outer Space).  Hatcher and Vogtmann showed that the edge splitting and free factor complexes -- at least the $\Aut$ versions of them, where we do not identify objects which differ by conjugation -- are both $(n-2)$-spherical \cite{HatcherVogtmann:CerfTheory,HatcherVogtmann:Complex_of_free_factors} (again, Hatcher and Vogtmann use the terminology `free factorization complex' in place of `edge splitting complex'). It seems to be an open question whether the $\Out$ versions of these complexes are also spherical.  To study $\Out$, Guirardel \cite{Guirardel} has introduced a notion of intersection form for two actions of $\Out$ on $\R$-trees.  Behrstock, Bestvina, and Clay \cite{BehrstockBestvinaClay} used Guirardel's intersection form to describe the effect of applying \emph{fully irreducible} automorphisms without periodic conjugacy classes to vertices in $\ES$.  They also discuss the edge splitting complex (therein called the splitting complex though HNN extensions are not allowed, as in \cite{kapovich_l:analogues_of_curve_complex}), and a related complex called the \emph{subgraph complex}.  Kapovich and Lustig \cite{Kapovich:Currents,Lustig:IntersectionForm} have also introduced an intersection form (distinct from Guirardel's), inspired by the work of Bonahon \cite{Bonahon}.  Kapovich and Lustig have shown that $\ES$ and $\FFZ$, as well as their intersection graph and 6 other related graphs, all have infinite diameter \cite{kapovich_l:analogues_of_curve_complex}. Recently, Yakov Berchenko-Kogan \cite{BerchenkoKogan} characterized vertices of distance 2 apart in the \emph{ellipticity graph}, a graph quasiisometric to $\FFZ$, using Stallings foldings.  This effectively characterizes adjacent vertices in $\FFZ$.  Very recently, Day and Putman~\cite{day_p:partial_bases} proved that another curve complex analogue, the \emph{complex of partial bases}, is simply connected.  The 1-skeleton of this complex is called the primitivity graph in~\cite{kapovich_l:analogues_of_curve_complex}, where it is also claimed that this graph is quasiisometric to the free factorization graph $\FFZ$. Aramayona and Souto have shown that $\Out$ is precisely the group of simplicial automorphisms of the free splitting complex $\FSC$ \cite{AramayonaSouto}.

None of these spaces are yet known to be hyperbolic.  Bestvina and
Feign \cite{BestvinaFeighn} have shown that, for any finite set $S$
of \emph{fully irreducible} outer automorphisms (see the paper for
definitions), there exists a hyperbolic graph $X$ with an isometric
$\Out$ action such that for any $\phi \in S$, $\phi$ acts with
positive translation length on $X$.  However, while these graphs are
hyperbolic, they depend on the choice of the finite set $S$.  The
results of Bestvina and Feighn imply that for a fully irreducibly
outer automorphism $\phi$, the maps $\Z \to \ES$ and $\Z \to
\FFZ$ given by $n \mapsto \phi^n v$ for any vertex $v$ is a
quasiisometric embedding.  Behrstock, Bestvina, and Feighn
\cite{BehrstockBestvinaClay} state that ``there is a hope that a
proof of hyperbolicity of the curve complex generalizes to the
[edge splitting] complex''.  However, we intend to prove:

\begin{theoremquasiflat}
For $n > 2$, the space $\ES$ (and hence $\ESC$) contains a quasiisometrically embedded copy of $\R^m$ for every $m\geq 1$.
\end{theoremquasiflat}

Our proof relies on attaining an understanding of distances in
$\ES$.  To do so, we associate vertices of $\ES$ with bases of
$\F$.  With this association, we are able to completely characterize
(up to distance 4) the length of a path in $\ES$ via a simple
algebraic notion which we call \emph{number of index changes}.  This
characterization is made precise in Theorem \ref{thm:geodesics} and
the preceding discussion.

To utilize this translation from geometry to algebra, we then
introduce an algebraic notion of complexity of a basis, which we
call \emph{$i$-length}.  The notion of $i$-length is itself based
roughly on having many subwords of elements of the basis with
complicated Whitehead graphs. Our techniques, in turn, use a theorem
of Stallings (see Section \ref{sec:ilength} for details). The bulk
of this paper aims to translate this $i$-length notion of how
complicated a basis is into a lower bound on distances between
vertices in $\ES$, as shown in the following theorem:

\begin{theoremlowerbound}
Let $\bx$ be a basis of $\F$, expressed in terms of a fixed standard
basis $\ba$.  The distance between a vertex of $\ES$ associated
to $\ba$ and one associated to $\bx$ is at least $\frac{|\bx|_i}{24}
- 1$, where $|\bx|_i$ is the $i$-length of $\bx$.
\end{theoremlowerbound}

As immediate corollaries of Theorem \ref{thm:quasiflat}, we obtain:
\begin{corollarynothyperbolic}
The space $\ES$ is not Gromov hyperbolic.
\end{corollarynothyperbolic}

In other words, $\ES$ is not the `correct' curve complex analogue for $\Out$.  This shows that the `hope' of \cite{BehrstockBestvinaClay} is a false one, at least for the edge splitting graph.  Indeed, it might be expected that the edge splitting graph is not hyperbolic:  edge splittings correspond to separating spheres in the sphere complex.  But in the mapping class group world, the subcomplex of the curve complex induced by only allowing separating curves is itself not hyperbolic \cite{Schleimer:CurveComplex}.

\begin{corollaryinfinitedimension}
The space $\ES$ has infinite asymptotic dimension.  The dimension
of every asymptotic cone of $\ES$ is infinite.
\end{corollaryinfinitedimension}

To the authors' knowledge, this is the only naturally defined space which has infinite asymptotic dimension and a natural cocompact group action of a group which is not known to have infinite asymptotic dimension.  Thompson's group $F$ acts on a cube complex with arbitrary-dimensional quasiflats \cite{Farley}, but has infinite asymptotic dimension (moreover, it is proved in~\cite{DranishnikovSapir} that $F$ has exponential dimension growth). Via private communication, Moon Duchin claims that the Cayley graph of $\Z$ with respect to the infinite generating set consisting of powers of $2$ has arbitrary-rank quasiflats.  Thus, we have a group with finite asymptotic dimension acting on a space with infinite asymptotic dimension.  However, this action is not cocompact:  the quotient is a graph with one vertex and infinitely many edges.  Both the mapping class group \cite{BestvinaBroombergFujiwara:asdimMCG} and arithmetic groups \cite{Ji:AsDimArithmetic} have finite asymptotic dimension, so the analogy between $\Out$ and these groups suggests that $\Out$ may in fact have finite asymptotic dimension.

There is a further interesting consequence of Theorem
\ref{thm:quasiflat}. There is a natural map $id^*$ from $\ES$ to
$\FFZ$ induced by the identity map on the vertex set. This map
$id^*$ is $1$-Lipshitz: if two free factorizations have a common
refinement, then any nontrivial elliptic element of the common
refinement will have translation length $0$ on both of the
corresponding Bass-Serre trees.  The quasiflats described in the
proof of Theorem \ref{thm:quasiflat} are in fact such that, for
every quasiflat, there exists a common elliptic element such that
every vertex in that quasiflat has a representative where one factor
contains the common elliptic element.  Thus,

\begin{corollarynotqi}
The map $id^*\colon \ES\to\FFZ$ is not a quasiisometry. Moreover,
there is no coarsely $\Out$-equivariant quasiisometry between $\ES$
and $\FFZ$.
\end{corollarynotqi}

An analogous results hold true for the relationships between the free factorization graph $\ES$ and the free factor graph $\FF$ and between $\ES$ and the free splitting graph $\FS$. There is a natural (coarsely well-defined for $n>2$)
map $\Sigma\colon \ES\to\FF$ defined by sending a
vertex $[A * B]$ in $\ES$ to the vertex $[A]$ in $\FF$. Also there is a natural embedding $\imath\colon \ES\to\FS$ defined by sending a vertex $[A*B]$ in $\ES$ to the vertex $[A*B]$ in $\FS$, which is quasisurjection. However, neither of the above maps is a quasiisometry:

\begin{corollarynotqi2}
The maps $\Sigma\colon \ES\to\FF$ and $\imath\colon \ES\to\FS$ are not quasiisometries. Moreover, there is no coarsely $\Out$-equivariant
quasiisometry between $\ES$ and $\FF$, and between $\ES$ and $\FS$.
\end{corollarynotqi2}

The last corollary provides a negative answer to a question of Bestvina and Feighn (the first half of Question 4.4
in~\cite{BestvinaFeighn}).

This paper is organized as follows.  We begin in Section
\ref{sec:3interpretations} by describing three ways of viewing an
element of $\Aut$.  Being able to translate between these three
perspectives will be useful at various points in the later proofs.
In Section \ref{sec:spaces}, we describe how to view vertices in
$\ES$ and $\FFZ$ as pairs consisting of an element of $\Aut$
and a proper nonempty subset of $\{1,2,\ldots,n\}$ up to certain
identifications. This viewpoint allows us to interpret distances in
$\ES$ algebraically, in terms of elements of $\Aut$, culminating
in Theorem \ref{thm:geodesics}.

Most of the detail in the paper lies in Section \ref{sec:ilength}.
In this section, we introduce the notion of $i$-length.  For
technical reasons, we use three different notions of $i$-length:
fixing some basis $\ba$ of $\F$, we have \emph{simple} $i$-length
for abstract words over $\ba$, \emph{conjugate reduced} $i$-length
for subwords written over $\ba$ of some other basis of $\F$, and
\emph{full} $i$-length for bases of $\F$ themselves.  In Section
\ref{sec:ilength}, we describe properties of each of these notions
of $i$-length in turn. The section builds up to, and ends with,
Theorem \ref{thm:lowerbound}.

Finally in Section \ref{sec:geometry}, we relate the algebraic notion of $i$-length to distances in $\ES$, and use this relationship to prove Theorem \ref{thm:quasiflat} and its corollaries, described above.

The authors would like to thank Ilya Kapovich, Diane Vavrichek,
Keith Jones, Dan Farley, and Karen Vogtmann for useful conversations
on this material, and Mladen Bestvina, Matt Clay, Michael Handel,
and especially Lee Mosher for useful comments.

\section{Three interpretations of $\Aut$}\label{sec:3interpretations}

Fix a basis $\ba = (a_1, \dots, a_n)$ of $\F$, considered as an
ordered tuple.  The group of all automorphisms of $\F$ has many
interpretations.  For our purposes, we will use three of these
interpretations, as follows.

The first interpretation of $\Aut$ is as in bijective correspondence with the set of ordered bases of $\F$. Consider a basis $\x = (x_1, \dots, x_n)$ of $\F$ as an ordered tuple.  As $\x$ is a basis, there exists an automorphism $\phi_\x$ which maps $\ba$ to $\x$; as automorphisms are uniquely specified by their action on a given generating set, $\phi_\x$ is unique.  Thus, $\Aut$ as a set is in bijective correspondence with the set
    $$\X := \{\x = (x_1,\dots, x_n)\in\F^n\ |\ \x \hbox{~is an ordered basis}\}.$$

The second interpretation of $\Aut$ is as products of \emph{elementary Nielsen automorphisms}.  Nielsen \cite{nielsen:aut_fn} described a generating set for $\Aut$ consisting of four types of generators:

\begin{definition}
An \emph{elementary Nielsen automorphism} is an automorphism of $\F$
for which there exist indices $i, j$ such that $i \neq j$, $a_k \mapsto a_k$
for $k \neq i, j$, and one of the following four possibilities holds:\\
$\begin{array}{rrl}
(1)&    s_{ij}:\qquad&  a_i \leftrightarrow a_j\\
(2)&    t_i:\qquad& a_i \mapsto a_i^{-1}\\
(3)&    a_{ij}:\qquad&  a_i \mapsto a_ia_j\\
(4)&    a_{ij}^{-1}:\qquad&  a_i \mapsto a_ia_j^{-1}\\
\end{array}$\\
\end{definition}
The group operation in $\Aut$ with respect to Nielsen automorphisms is function composition, where automorphisms are composed as functions, right-to-left.  Note a Nielsen automorphism $\phi$ acts on the Cayley graph of $\Aut$ via the usual left action.  We can interpret this action on the vertices of the Cayley graph in terms of the correspondence between $\Aut$ and $\X$:  an automorphism $\phi$ acting on a basis $\x \in \X$ has image $\phi(\x) = (\phi x_1, \dots, \phi x_n) = \phi \circ \phi_\x (\ba)$.

The third interpretation of $\Aut$ is as the group of Nielsen
transformations.  A \emph{Nielsen transformation} is an action on
the set of ordered bases of $\F$ (that is, on $\Aut$, by the first
interpretation) which may be decomposed as a product of
\emph{elementary Nielsen transformations}.  These elementary Nielsen
transformations are free-group analogues of the elementary row
operations in $GL_n(\Z) = Aut(\Z^n)$, and, in fact, induce the
elementary row operations under the abelianization map $\F \to
\Z^n$.  There are four kinds of elementary Nielsen transformations:

\begin{definition}
An \emph{elementary Nielsen transformation} is a map on the set of ordered
bases $\X = \{\x = (x_1, \dots, x_n)\}$ of $\F$ for which there exist
indices $i, j$ such that $i \neq j$, $x_k \mapsto x_k$ for $k \neq i, j$,
and one of the following four possibilities hold:\\
$\begin{array}{rrl}
(1)&    \sigma_{ij}: \qquad&    x_i \leftrightarrow x_j\\
(2)&    \tau_i: \qquad&     x_i \mapsto x_i^{-1}\\
(3)&    \alpha_{ij}: \qquad&    x_i \mapsto x_ix_j.\\
(4)&    \alpha_{ij}^{-1}:\qquad&  x_i \mapsto x_ix_j^{-1}\\
\end{array}$\\
An elementary Nielsen transformation of Types (3) and (4) are called
\emph{transvections}.
\end{definition}
The group operation in $\Aut$ with respect to Nielsen
transformations is again composition, but transformations are
composed left-to-right.  Nielsen transformations act on $\X$ on the
right.

The isomorphism between the groups generated by Nielsen automorphisms and by Nielsen
transformations is clear:  the isomorphism is $s_{ij} \mapsto
\sigma_{ij}$, $t_i \mapsto \tau_i$, $a_{ij} \mapsto \alpha_{ij}$.
Thus, a word in elementary Nielsen transformations may be considered
as a word in Nielsen automorphisms, \emph{written in the same
order}, but with the order of composition reversed and the action on
$\X$ on the left instead of the right.

These three interpretations are different aspects of the same
concept:  the set $\X$ may be viewed as the vertices of the Cayley
graph of $\Aut$; elementary Nielsen automorphisms form a generating
set of $\Aut$ and their action on $\X$ corresponds to the left
action of this generating set on its Cayley graph. This is the action such that an automorphism $g$ takes a vertex $v$ to $gv$, and takes an edge connecting $v$ to $va$ to an edge connecting $gv$ and $gva$ for each generator $a$ of $\Aut$. Elementary
Nielsen transformations form the same generating set, but with the
action on $\X$ being an interpretation of the right action of the
generating set on the vertices of its Cayley graph. When restricted to the action of a generator $a$ of $\Aut$, it simply moves a vertex $v$ across the edge connecting $v$ to $va$ to the vertex $va$. However, this right action does not extend to the edges of the Cayley graph.

In his seminal paper \cite{nielsen:aut_fn}, Nielsen presented a
method for transforming a finite generating set for a subgroup of a
free group into a free basis for that subgroup using elementary
Nielsen transformations. Nielsen's method is essentially a finite
reduction process, at every step of which a Nielsen transformation
is used to `simplify' the finite generating set.  In
Lemma~\ref{lem:kis0} we will apply this process to the bases of $\F$
and will use the following fact, whose proof follows from the proof
of Theorem 3.1 in~\cite{MagnusKarassSolitar}.

\begin{proposition}
\label{prop:nielsen} For every basis $\bx$ of a free group $\F$
there is a sequence of elementary Nielsen transformations
$(\delta_j)$, $1\leq j\leq t$ taking the standard basis $\ba$ of
$\F$ to $\bx = \ba \delta_1 \dots \delta_t$ such that the sum of the
lengths (with respect to $\ba$) of elements in the intermediate
bases $\ba \delta_1 \dots \delta_j$ is a nondecreasing sequence.
\end{proposition}

\section{Vertices and edges in $\ES$}\label{sec:spaces}

We wish to view the spaces $\ES$ and $\FFZ$ on which $\Out$ acts in the
language of ordered tuples, so that we may apply the dictionary of
Section \ref{sec:3interpretations} equating tuples, Nielsen
automorphisms, and Nielsen transformations.

We begin with an observation on elements of $\Out$.  An element of $\Out$ is a
coset of $\Aut$ with respect to the subgroup $\Inn$.  As such,
an element of $\Out$ may be represented by many different
$n$-tuples.  In general, we think of an element of $\Out$ as a tuple
up to conjugation.

Now consider the graphs $\ES$ and $\FFZ$.  These graphs have
the same vertex set:  vertices correspond to free factorizations of
$\F$ into two nontrivial factors up to conjugation.  We wish to
interpret an arbitrary free factorization of $\F$ into two factors
as a tuple, together with an index set, up to certain equivalences.
Let $\mathcal{S}$ denote the set of all proper nonempty subsets of
$\{1, \dots, n\}$. We will call an element of $\mathcal S$ an
\emph{index set}. Then a tuple $\bx = (x_1, \dots, x_n)$ together
with some index set $S \in \mathcal{S}$ yields a free factorization
of $\F$ as $\langle\bx_S\rangle
* \langle\bx_{\ol{S}}\rangle$, where $\bx_S := \{a_i \in \bx | i \in
S\}$ and $\ol{S} := \{1, \dots, n\} - S$.  Every free factorization may
be represented as a tuple/index set pair, but a given free factorization
may be represented by multiple tuple/index set pairs:  any
tuple/index set pairs which differ by a self-map of $\Aut \times
\mathcal{S}$ preserving the associated free factorization up to
conjugation should be identified.

Every such map can be written as a composition of four types of
self-maps, defined by their action on $(\bx,S)\in\Aut\times\mathcal
S$ as follows:
\begin{enumerate}
\item conjugation of $\bx$ without changing $S$,
\item permutation of $\{1,
\dots, n\}$ applied to both $\bx$ and $S$,
\item exchanging $S$ for $\ol{S}$ and leaving $\bx$ unchanged,
\item applying transformation $\phi$ of $\Aut$ fixing the free factors in the factorization
setwise (i.e. $\langle\bx_S\rangle =\langle(\bx\phi)_S\rangle$ and
$\langle\bx_{\ol{S}}\rangle =\langle(\bx\phi)_{\ol{S}}\rangle$)
without changing $S$.
\end{enumerate}

The transformation $\phi$ in the last item is called an
\emph{$S$-transformation}. If there exists $S \in \mathcal{S}$ such
that $\phi$ is an $S$-transformation, we call $\phi$ an
$\mathcal{S}$-transformation.

Note that any self-map from the group mentioned above may be
realized as composition of the form $m_1m_2m_3m_4$, where $m_i$ is a
self-map of type $(i)$.

With this interpretation of vertices of $\ES$, consider edges of
$\ES$.  Two vertices represented by $A_1 * B_1$ and $A_2 * B_2$
of $\ES$ are adjacent if there exists a common refinement of
conjugates of the free factorizations.  Such a common refinement is
of the form $A*C*B$, where for some elements $g$ and $h$ of $\F$ we
have $A_1^g = A*C$, $B_1^g = B$, $A_2^h = A$, and $B_2^h = C*B$.
Without loss of generality we can assume that $h$ is trivial. If
$(\bx,S)$ is the vertex corresponding to the free factorization
$A_1*B_1$, then $(\bx^g,S)$ represents the same vertex of $\ES$
and the refinement $A_1^g = A*C$ corresponds to applying to $\bx^g$
a transformation $\phi$ of $\F$ taking $\bx^g_S$ to a basis for $A$
union a basis for $C$ and preserving $B_1^g$.  Note $\phi$ fixes
both $\langle\bx^g_S\rangle$ and $\langle\bx^g_{\ol{S}}\rangle$, and
so is an $S$-transformation. Then changing $(A*C)*B$ to $A*(C*B)$
simply corresponds to subtracting from $S$ the indices of elements
in $\phi(\bx^g)$ corresponding to a basis for $C$.  Of course, by
exchanging $S$ for $\ol{S}$, we could have instead added elements to
$S$, which corresponds to subtracting elements from $\bar{S}$. Thus,
changing $(A*C)*B$ to $A*(C*B)$ corresponds to replacing $S$ with a
proper subset of either $S$ or $\bar{S}$. We call index sets $S$ and
$S'$ from $\mathcal S$ \emph{compatible} if either $S'$ or
$\overline{S'}$ is a proper subset of either $S$ or $\bar S$.

Thus, up to conjugation, all edges from the vertex corresponding to
$(\bx,S)$ are precisely characterized by a transformation fixing
$\langle\bx_S\rangle$ and $\langle\bx_{\ol{S}}\rangle$, followed by
replacing $S$ with a compatible element of $\mathcal{S}$.  We have
shown:

\begin{lemma}
The set of edges in $\ES$ from a vertex $(\bx,S)$ is determined
by:  a conjugation of $\bx$, an $S$-transformation, and a choice of
new index set compatible with $S$.
\end{lemma}

An edge path $p$ from the vertex represented by $(\bx,S_0)$ in
$\ES$ is described by a sequence:  a conjugation $\gamma_0$ of
$\bx$, an $S_0$-transformation $\phi_0$, a change of index set to
$S_1$, a conjugation $\gamma_1$ of $\bx^{\gamma_0}\phi_0$, an
$S_1$-transformation $\phi_1$, a change of index set to $S_2$, etc.

On such an edge path $p$, for any $i$, let $(\bx(i),S_i)$ be a
representative of the vertex on the edge path immediately before
$\phi_i$ is to be applied. Then by construction we have

\[\bx(i)=
\bigl(\ldots\bigl((\bx^{\gamma_0}\phi_0)^{\gamma_1}\phi_1\bigr)^{\gamma_2}\ldots\bigr)^{\gamma_{i-1}}\phi_{i-1}=
(\bx \phi_0 \phi_1 \dots
\phi_{i-1})^{\bigl(\ldots\bigl((\gamma_0)\phi_0\gamma_1\bigr)\phi_1\gamma_2\ldots\bigr)\phi_{i-1}}.\]
As vertices of $\ES$ are only defined up to conjugation, we may
assume without loss of generality that all of the conjugators
$\gamma_i$ are trivial and
\[\bx(i) = \bx \phi_0 \phi_1 \dots
\phi_{i-1}.\] The set $S_i$ is not determined uniquely by $\phi_i$,
as $\phi_i$ may be an $S$-transformation for many index sets $S$.
However, for any such $S$, the vertex $(\bx(i),S)$ is of distance at
most $2$ away from each of the vertices $(\bx(i-1),S_{i-1})$,
$(\bx(i),S_i)$, and $(\bx(i+1),S_{i+1})$ in $\ES$, as follows.
That $(\bx(i),S)$ is distance at most $2$ from $(\bx(i-1),S_{i-1})$
follows from applying $\phi_{i-1}$ to $(\bx(i-1),S_{i-1})$ and then
changing the index set to $S$, which requires one edge in $\ES$
if $S_{i-1}$ and $S$ are compatible and 2 edges otherwise. That
$(\bx,S)$ is distance at most $2$ from $(\bx(i),S_i)$ follows from
applying the identity transformation to $(\bx(i),S_i)$ (note the
identity transformation is indeed an $S_i$-transformation) and then
changing index set to $S_i$.  Finally, for the vertex
$(\bx(i+1),S_{i+1})$, since $\phi_i$ is an $S$-transformation,
$(\bx(i+1),S_{i+1})$ is the vertex obtained from $(\bx,S)$ by
applying $\phi_i$ to $\bx(i)$ and then changing the index set to
$S_{i+1}$.

Thus, up to distance $2$ at every vertex on the path $p$, the path
$p$ is determined by the sequence of transformations $\phi_0,
\phi_1, \dots, \phi_k$ (see Figure \ref{fig:ladder}).  Note that we
may reverse this procedure: take a sequence of transformations
$\phi_0, \phi_1, \dots, \phi_k$ such that each $\phi_i$ is an
$\mathcal{S}$-transformation, choose any $S_i' \in \mathcal{S}$ such
that $\phi_i$ is an $S_i'$-transformation, and obtain an edge path
in $\ES$, which is uniquely defined up to distance $2$ at each
vertex.

\begin{figure}

\hspace{-.5in}
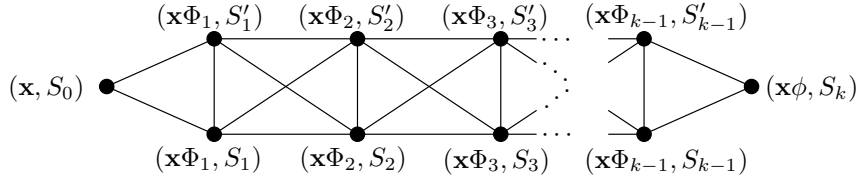

\caption{Shown are two edge paths from the vertex $(\bx,S_0)$ to the
vertex $(\bx\phi,S_k)$ in $\ES$ using the notation of this
section.  Here, we let $\Phi_i := \phi_0\dots \phi_{i-1}$ denote the
composition of $\mathcal{S}$-transformations, where $\phi_i$ is an
$S_{i-1}$-transformation.  Thus, $\bx(i) = \bx\Phi_i =
\bx\phi_0\dots \phi_{i-1}$. The lower path represents the edge path
described in the text, and is represented by the sequence of
transformations $\phi_0, \dots, \phi_k$. The upper path represents
an edge path reconstructed from the $\mathcal{S}$-transformations
$\phi_0, \dots, \phi_k$ by, for each $i = 1, \dots, k-1$, choosing
an arbitrary index set $S_i'$ compatible with $S_{i-1}'$ such that
$\phi_i$ is an $S_i'$-transformation.  Horizontal edges in the
figure are edges in $\ES$, and vertical and diagonal edges mean
that the distance between two vertices in $\ES$ is at most
2.}\label{fig:ladder}

\end{figure}

A geodesic in $\ES$ is then easy to describe.  A geodesic, up to
distance 2 at each vertex, is an edge path $\phi_0, \phi_1, \dots,
\phi_k$ such that the transformation $\phi = \phi_0\phi_1\dots
\phi_k$ is not a product of fewer than $k+1$
$\mathcal{S}$-transformations with the property that the neighboring
transformations are $\mathcal S$-transformations with respect to
compatible index sets.

For a given word $w$ in the generating set for $\Aut$ consisting of
elementary Nielsen transformations and the identity transformation,
we say that $w$ has \emph{at most $k$ index changes} if $w$ may be
expressed as a product of $k+1$ disjoint subwords, each of which is
an $\mathcal{S}$-transformation and the neighboring subwords are
$\mathcal{S}$-transformations with respect to compatible index sets.
If $k$ is minimal over all such products, we say $w$ \emph{requires
$k$ index changes}.  Since the product of $S$-transformations is an
$S$-transformation, we can rephrase the preceding paragraph in the
form of the following Theorem.

\begin{theorem}\label{thm:geodesics}
A geodesic in $\ES$ is represented by a product of $\mathcal
S$-transformations with the minimal number of index changes.
Moreover, a geodesic of length $k$ in $\ES$ requires from $k-4$,
to $k$ index changes.
\end{theorem}

We will use this characterization to describe lower bounds on
distances in $\ES$ based on properties of the associated
transformations in Section \ref{sec:ilength}.

We end this section by noting that there is a similar
characterization of roses in the spine $K_n$ of outer space as
tuples, up to conjugation and signed permutation (the signed
permutations correspond to graph isomorphisms).  With this
interpretation, there are canonical Lipschitz maps from $K_n$ to
$\ES$ to $\FFZ$.  It is also worth noting that the
quasiisometry between $\Out$ and $K_n$ may be stated in this
language:  Let $K_n'$ be the graph whose vertices are the marked
roses of $K_n$ and whose edges correspond to marked roses lying on a
common $2$-cell in $K_n$.   Then $K_n$ is 2-biLipschitz equivalent
to the graph $K_n'$, and $K_n'$ is biLipschitz equivalent to the
Cayley graph of $\Out$ with respect to the generating set of
elementary Whitehead transformations:  $K_n'$ is the Schreier graph
of $\Out$ with respect to this generating set and the finite
subgroup of signed permutations.

\section{The Notion of $i$-Length}\label{sec:ilength}

In this section, we define the notion of \emph{$i$-length} and
analyze its properties.  This notion is an algebraic tool that will
be used to estimate distances in $\ES$.  We use the concept of
$i$-length to refer to a measure of complexity of 3 different kinds
of objects:  abstract words in the generators of $\F$, subwords of
bases of $\F$, and bases of $\F$ themselves.  Our 3 concepts of
$i$-length are:  \emph{simple} $i$-length, \emph{conjugate reduced}
$i$-length, and \emph{full} $i$-length, respectively.  We use simple
$i$-length to define conjugate reduced $i$-length, and conjugate
reduced $i$-length to define full $i$-length.  After defining the
three notions of $i$-length, we will analyze the properties of each
in turn.

Throughout this section, we fix a standard basis $\ba = \{a_1, \dots, a_n\}$ of $\F$ once and for all.

\subsection{Defining $i$-Length}~\\
We motivate our definition of $i$-length with an example.

Let $H := \la a_1, \dots, a_{n-1}\ra$ denote the subgroup of $\F$ of
rank $n-1$ corresponding to ignoring the generator $a_n$.  Consider
the vertex $v_0 := [H * \la a_n \ra]$ as a basepoint in $\ES$,
and think about moving in $\ES$ to the vertex $v = [H * \la a_n h
\ra]$, where $h$ is an arbitrary element of $H$.  Let $d$ denote the
distance between $v_0$ and $v$ in $\ES$.

If $h$ is nontrivial, then $v \neq v_0$, as there is clearly no way
of using conjugation to remove occurrences of all elements of $H$
from the second factor of any representative of $v$. Moreover, as
$v_0$ and $v$ both have the same index set, by Theorem
\ref{thm:geodesics}, when $h$ is nontrivial we have $d \geq 2$.

If $h$ is a primitive element in $H$, then $d = 2$, as follows. Let
$h_2, \dots, h_{n-1}$ denote elements of $H$ such that $\{h, h_2,
\dots, h_{n-1}\}$ forms a basis for $H$.  Then $\la h, h_2, \dots,
h_{n-1} \ra * \la a_n \ra$ is a representative of $v_0$, and $\la h,
h_2, \dots, h_{n-1} \ra * \la a_n h \ra$ is a representative of $v$.
Thus, $[\la h_2, \dots, h_{n-1} \ra * \la h, a_n \ra]$ is a vertex
which is adjacent to both $v$ and $v_0$.

\begin{figure}[!h]
\epsfig{file=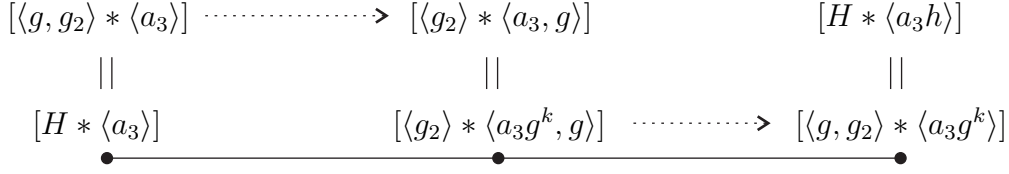} \caption{A path of length 2
in $\ES[3]$.} \label{fig:dist2}
\end{figure}

If $h$ is a power of a primitive element in $H$, the same argument
again shows that $d = 2$.  Figure~\ref{fig:dist2} shows the path of
length 2 connecting $[H * \la a_n \ra]$ and $[H * \la a_n h \ra]$
for $n=3$, $h=g^k$, where $g$ is primitive in $H$ and $g_2$ is some
coprimitive with $g$ element such that $\la g,g_2\ra=H$. Repeating
the above argument shows that, if we know that $h$ is a product of
$j$ powers of primitive elements in $H$, then $d \leq 2j$.  To
obtain a lower bound on $d$, we need to at least minimize $j$. Thus,
we need to consider how to detect how many powers of primitives are
needed to form $h$.

One property of a (power of a) primitive element $h$ of $H$ is a
classical result of Whitehead, which states that the \emph{Whitehead
graph} of $h$, considered as a reduced word in the alphabet $(\ba -
\{a_n\})^{\pm 1}$, must have a \emph{cut vertex}, defined as
follows.

\begin{definition}[Whitehead graph]
For a set of freely reduced words $\bx = \{x_1, \dots, x_k\}$ in the
alphabet $\ba \cup \ba^{-1}$, define the \emph{Whitehead graph}
$\Gamma_\ba(\bx)$ as follows.  The set of vertices
of $\Gamma_\ba(\bx)$ is identified with the set $\ba \cup \ba^{-1}$.  For
every $x_i \in \bx$ of length $n$, $x_i$ contributes exactly $n-1$
edges to $\Gamma_\ba(\bx)$, one for each pair of consecutive letters in
$x_i$.  The edge added for a $a_ia_j$ is from the vertex $a_i$ to
the vertex $a_j^{-1}$.  The \emph{augmented Whitehead graph}
$\hat\Gamma_\ba(\bx)$ is the Whitehead graph $\Gamma_\ba(\bx)$
together with an additional edge for each $x_i \in \bx$, from the
last letter of $x_i$ to the inverse of the first letter.  In
particular, a word $x_i = a_j$ of length 1 contributes exactly one
edge, from $a_j$ to $a_j^{-1}$, to $\hat\Gamma_\ba(\bx)$.  For a
single word $w$, we abuse notation and write $\Gamma_\ba(w)$ for
$\Gamma_\ba(\{w\})$ and $\hat\Gamma_\ba(w)$ for $\hat\Gamma_\ba(\{w\})$.
\end{definition}

If $\bx$ is cyclically reduced and linearly independent, then the Whitehead graph of a set of freely reduced words $\bx$ is graph-isomorphic to the link of the unique vertex in the presentation 2-complex of the group $\F/\la\la\bx\ra\ra$ generated by $a_1, \dots, a_n$ with relations $x_1, \dots, x_k$.

Note that a Whitehead graph (or augmented Whitehead graph) may have
multiple edges. Loops at a vertex may appear only in an augmented Whitehead
graph and if and only if at least one of the words in $\bx$ is not
cyclically reduced. An example of the augmented Whitehead graph,
namely $\hat\Gamma_{\{a_1,a_2,a_3,a_4\}}(a_2^2a_3^2a_4^2)$, is shown
in Figure~\ref{fig:whitehead}.

\begin{figure}[!h]
\epsfig{file=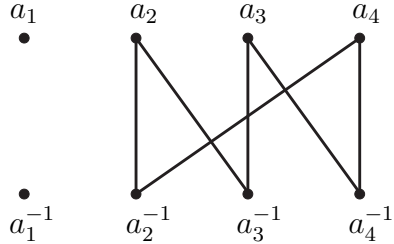}
\caption{Augmented Whitehead graph $\hat\Gamma_{\{a_1,a_2,a_3,a_4\}}(a_2^2a_3^2a_4^2)$}
\label{fig:whitehead}
\end{figure}

\begin{definition}[cut vertex]
A \emph{cut vertex} $v$ of a graph $\Gamma$ is a vertex such that $\Gamma = \Gamma_1 \cup \Gamma_2$, where $\Gamma_1$ and $\Gamma_2$ are nonempty subgraphs and $\Gamma_1 \cap \Gamma_2 = \{v\}$.  If $\Gamma$ is disconnected, then all of its vertices are cut vertices.
\end{definition}

Whitehead proved \cite{whitehead:graph} that the augmented Whitehead graph of a basis of a free group has a cut vertex.  Note that a power of a primitive has the same augmented Whitehead graph as the given primitive, so the augmented Whitehead graph of a power of a primitive must also have a cut vertex.  The converse is, of course, not true -- for example, $aba^3b$ is not primitive in $F_2$ -- but of course the contrapositive is:  having an augmented Whitehead graph with no cut vertex implies the element is not a primitive or a power of a primitive.

For our purposes, we will need a generalization of Whitehead's theorem due to Stallings \cite{stallings:handlebodies}, so we state it now.  A subset $S$ of $\F$ is called \emph{separable} if there is a free factorization of $\F$ with two factors such that each element of $S$ can be conjugated into one of the factors.  In particular, a set is separable if its elements can be conjugated (possibly by different conjugators) to the elements of some basis of $\F$.  Thus, a basis (and the cyclic reduction of a basis) is always separable.

\begin{theorem}[\cite{stallings:handlebodies}]\label{thm:cutvertex}
If $\bx$ is a separable set in $\F$, then there is a cut vertex in $\hat\Gamma_\ba(\bx)$.
\end{theorem}

Now consider our motivating example of the distance $d$ between $v_0 = [H * \la a_n \ra]$ and $v = [H * \la a_n h \ra]$ in $\ES$.  Na\"ively, we could hope that if we could break up $h$, considered as a reduced word, into $k$ subwords such that each subword had an augmented Whitehead graph with no cut vertex, then $d$ might be bounded from below by a function of $k$.  However, it may not be the case that such a decomposition of $h$ `breaks' $h$ in the places corresponding to the most efficient way of decomposing it as a product of powers of primitives:  a given primitive might contribute to one or more of the subwords.  But Whitehead's theorem does \emph{not} say that the (augmented) Whitehead graph of any subword of a primitive will have a cut vertex.  Indeed, a primitive element conjugated by an arbitrary word will still be primitive, and the only reason its augmented Whitehead graph will have a cut vertex will be from the single self-loop contributed by the last and first letters.  If the primitive element is cyclically reduced, then we may claim that the (non-augmented or augmented) Whitehead graph of any subword will have a cut vertex, but not otherwise.

The notions of $i$-length are defined precisely to deal with this delicate effect of conjugation.  Simple $i$-length ignores
conjugation completely, looking only at the non-augmented Whitehead graph of a word and its subwords.  Conjugate reduced $i$-length takes all possible conjugations of the subwords of a word into account.  Full $i$-length then uses conjugate reduced $i$-length to measure the complexity of an entire basis.

We are almost ready to give the definitions of $i$-length, but we need one minor piece of notation to proceed.

\begin{notation}
Elements of $\F$ are equivalence classes of words in the alphabet
$\ba \cup \ba^{-1}$ under free reduction.  For two words $w_1$ and
$w_2$ in this alphabet, we write $w_1 = w_2$ if they are equal as
words, and $w_1 =_r w_2$ if they are equal after free reduction,
i.e. as elements of $\F$.
\end{notation}

We now define the 3 notions of $i$-length.  The definition of full
$i$-length is somewhat nuanced, but the concept is based on simple
$i$-length.  The idea of simple $i$-length is straightforward:  it
records the maximal number of pieces a word can be broken into such
that the Whitehead graph of each piece has no cut vertex.

\begin{definition}[Simple $i$-length]
Fix an index $i \in \{1, \dots, n\}$.  Let $w$ be a word which
contains no occurrence of $a_i^{\pm 1}$.  The \emph{simple
$i$-length} of $w$, denoted $|w|_i^{simple}$, is the greatest number
$t$ such that $w$ is of the form $w_1w_2\dots w_t$, where
$\Gamma_{\ba-\{a_i\}}(w_j)$ has no cut vertex for each $j = 1, \dots, t$.  If $\Gamma_{\ba-\{a_i\}}(w)$ has a cut vertex, we define
$|w|_i^{simple}$ to be zero.
\end{definition}

It worth pointing out that in the above definition we use standard
Whitehead graph, not the augmented one.

Conjugate reduced $i$-length additionally takes conjugation into
account.

\begin{definition}[Conjugate reduced $i$-length]
Fix an index $i \in \{1, \dots, n\}$.  Let $w$ be a word which
contains no occurrence of $a_i^{\pm 1}$ (thought of as a subword of
another word in the alphabet $\ba^{\pm 1}$).  Then $w$ has
\emph{conjugate reduced $i$-length at most $k$} if there exist
freely reduced words $v_1, \dots, v_l, u_1, \dots, u_l$ such that:
\begin{enumerate}
\item $w =_r v_1^{u_1}v_2^{u_2}\dots v_l^{u_l}$, where $v_j^{u_j} := u_j^{-1}v_ju_j$, and
\item $k = (l-1) + |v_1|_i^{simple} + \dots + |v_l|_i^{simple}$.
\end{enumerate}
The decomposition of $w$ as $v_1^{u_1}v_2^{u_2}\dots v_l^{u_l}$ is
called a \emph{decomposition}, and $k$ is the \emph{conjugate
reduced $i$-length associated to the decomposition}.  If the
associated $k$ is minimal among all such decompositions, the
decomposition is called \emph{optimal}, and $k$ is called a
conjugate reduced $i$-length of $w$ and denoted by $|w|_i^{cr}$. The
number $l$ of factors of the form $v_j^{u_j}$ in the decomposition
is called the \emph{factor length} of the decomposition.
\end{definition}

We are now ready to define the (full) $i$-length of a basis for $\F$
(or more generally a set of words).  Given a basis $Y$, we
essentially measure the maximal conjugate reduced $i$-length of any
subword of any element of $Y$.  However, we must be very careful to
properly account for conjugation.  We do so as follows.

Let $\by$ be a set of reduced words in the alphabet $\ba^{\pm 1}$.
Let $\tilde{\by}$ denote the set of elements of $\by$ after each of
them has been cyclically reduced.  Define $w_L = w_L(\by)$ to be the
longest word in the alphabet $(\ba - \{a_i\})^{\pm 1}$ such that
every occurrence of $a_i$ in every $\tilde{y} \in \tilde{\by}$ is
cyclically preceded by $w_L$ and every occurrence of $a_i^{-1}$ is
cyclically followed by $w_L^{-1}$ (note $w_L$ could be trivial).
Similarly, let $w_R = w_R(\by)$ be the longest word in $(\ba -
\{a_i\})^{\pm 1}$ such that every occurrence of $a_i$ in every
$\tilde{y}\in \tilde{\by}$ is cyclically followed by $w_R$, every
occurrence of $a_i^{-1}$ is cyclically preceded by $w_R^{-1}$, and
no such occurrence of $w_R$ intersects any such occurrence of $w_L$
(again, $w_R$ could be trivial).  Let $\alpha' = \alpha'_\by$ be the
automorphism of $\F$ which maps $a_i$ to $w_L^{-1}a_iw_R^{-1}$. Let
$w_C = w_C(\by)$ be the longest word in $(\ba - \{a_i\})^{\pm 1}$
such that, in $\alpha' \tilde{\by}$, every occurrence of $a_i^k$
either: (a) occurs by itself as an element of $\alpha'\tilde{\by}$ or (b) appears cyclically conjugated by $w_C$, so that $a_i^k$ is cyclically preceded by $w_C^{-1}$ and cyclically followed by $w_C$.  If every
occurrence of $a_i^k$ occurs by itself, we declare that $w_C$ is
trivial.

If $\by$ is a singleton $\{y\}$, we abuse notation and write $y$
instead of $\{y\}$ when applying any function in this subsection.

Let $\alpha = \alpha_\by$ be the automorphism of $\F$ which maps
$a_i$ to $w_L^{-1}w_Ca_iw_C^{-1}w_R^{-1}$.  Thus, $w_L(\alpha \by) =
w_R(\alpha \by) = 1$ are trivial.  The preimage of $a_i^k$ under
$\alpha$, after free reduction, is $w_C^{-1}~(w_La_iw_R)^k~w_C$.
Note that $w_C^{-1}$ and $w_C$ canceled between adjacent occurrences
of $a_i$, but that this is the only free cancelation which occurs in
$\alpha^{-1}(\alpha \by)$.  Note $\alpha \by$ may not be cyclically reduced.

An \emph{$i$-chunk} of a word $y$ in the alphabet $\ba^{\pm 1}$ is a
cyclic subword of $\tilde{y}$ (here again, $\tilde y$ denotes the
result of a cyclic reduction of $y$) which contains no $a_i^{\pm 1}$
and is maximal among such subwords ordered by inclusion.  By
definition, every $i$-chunk of $y$ begins with either $\wR{y}$ or
$\left(\wL{y}\right)^{-1}$, and ends with either $\wL{y}$ or
$\left(\wR{y}\right)^{-1}$.

For example, in the set $\by = \{a_2^{-1}a_3a_4a_1a_4a_2a_3,~
a_4^{-1}a_1^{-1}a_4^{-1}a_3^{-1}\}$, we have $\tilde{\by} = \by$.
For $i = 1$, $w_L(\by) = a_3a_4$, $w_R(\by) = a_4$, and $w_C(\by) =
a_2$.  Thus, $\alpha(a_1) = a_4^{-1}a_3^{-1}a_2a_1a_2^{-1}a_4^{-1}$,
so that $\alpha(\by) = \{a_1a_3,~ a_2a_1^{-1}a_2^{-1}\}$.

\begin{definition}[Full $i$-length]
Fix an index $i \in \{1, \dots, n\}$.  Let $\by$ be a set of words in the alphabet $\ba^{\pm 1}$.  The (full) \emph{$i$-length} of $\by$ is
    $$|\by|_i := k(\by) + |\wR{\by}\wL{\by}|_i^{cr},$$
where $k(\by)$ is the maximal conjugate reduced $i$-length of an
$i$-chunk of $\alpha_{\by} y$ over all elements $y \in \by$.
\end{definition}

For example, let $w = a_1^2a_2^2\dots a_{n-1}^2a_1$.  Then
$|w|_n^{simple} = |w|_n^{cr} = 1$, and $|a_nw|_n = 1$.  We will
later see (in Corollary \ref{cor:arbitrarylength}) that $|a_nw^l|_n
\geq l/3-2$.

\subsection{Properties of Simple $i$-length}~\\
This subsection includes three simple lemmas that we will use in further proofs.

\begin{lemma}
\label{lem:simple_length}
Let $w$ be a freely reduced word in $\F$ which contains no occurrence of $a_i^{\pm 1}$ and let $u$ and $v$ be disjoint subwords of $w$. Then
    \[|w|_i^{simple}\geq |u|_i^{simple}+|v|_i^{simple}.\]
\end{lemma}

\begin{proof}
If $|u|_i^{simple}>0$  and $|v|_i^{simple}>0$ then consider the partitions of $u$ and $v$ into $|u|_i^{simple}$ and $|v|_i^{simple}$ pieces respectively.  These partitions induce a partition of $w$ into $|u|_i^{simple}+|v|_i^{simple}$ pieces, where the portions of $w$ disjoint from $u$ and $v$ are appended to the first and last pieces in the partitions of $u$ and $v$.  Note that appending will not brake the fact that the Whitehead graph of a piece does not have a cut vertex, because $w$ is freely reduced. If $|u|_i^{simple}=0$ (respectively, $|v|_i^{simple}=0$) then we similarly form a partition of $w$ into $|v|_i^{simple}$ (respectively, $|u|_i^{simple}$) pieces. In the case $|u|_i^{simple}=0$ and $|v|_i^{simple}=0$ the claim is trivial.
\end{proof}

\begin{lemma}
\label{lem:simple_length2}
Let $u$ and $v$ be freely reduced words which contain no occurrence of $a_i^{\pm 1}$  such that $w = uv$ is freely reduced.  Then
    \[|w|_i^{simple} \leq |u|_i^{simple}+|v|_i^{simple}+1.\]
\end{lemma}

\begin{proof}
If $|w|^{simple}_i=0$ then the claim is trivial. Otherwise let $w=w_1w_2\cdots w_k$ be a partition of $w$ realizing the simple $i$-length of $w$, and let $j$ denote the first index such that $w_j$ is not fully contained in $u$. This gives partitions $u=w_1w_2\cdots (w_{j-1}w_j')$ and $v=(w_j''w_{j+1})w_{j+2}\cdots w_k$ of $u$ and $v$ showing that
\[|u|_i^{simple}+|v|_i^{simple}\geq (j-1)+(k-j)=k-1=|w|_i^{simple}-1.\]
\end{proof}

\begin{lemma}
\label{lem:simple_conj}
If $w$ is a cyclically reduced word which contains no occurrence of $a_i^{\pm 1}$ and $w'$ is a cyclic conjugate of $w$, then
    \[|w|_i^{simple}-1\leq|w'|_i^{simple}\leq|w|_i^{simple}+1.\]
\end{lemma}

\begin{proof}
It is enough to show that cyclic conjugation cannot decrease the simple $i$-length by more than 1. Let $\iota(w)$ be the initial segment of $w$ such that $w' =_r w^{\iota(w)}$.  Let $w=w_1w_2\cdots w_k$ be the partition of $w$ realizing the simple $i$-length of $w$, and let $j$ denote the first index such that $w_j$ is not fully contained in $\iota(w)$.  Then $w'$ can be partitioned as ($w_j''w_{j+1})\dots w_kw_1\dots (w_{j-1}w_j')$ where $w_j = w_j'w_j''$.  Thus, $w'$ can be partitioned into at least $k-1$ subwords of nontrivial simple $i$-length, and the lemma follows.
\end{proof}

\subsection{Properties of Conjugate reduced $i$-Length}~\\
We now wish to describe some properties of conjugate reduced
$i$-length.  However, before we do so, we need to verify that
conjugate reduced $i$-length is not a trivial notion of complexity.
In this section, we show that there exist words of arbitrary
conjugate reduced $i$-length.  In the process, we develop a useful
lemma for working with $i$-length.  Then, we collect three short
lemmas which describe how conjugate reduced $i$-length is related to
simple $i$-length, and how conjugate reduced $i$-length behaves
under multiplication.

\begin{definition}[canceling pairs]
Let $w\in\F$ be arbitrary reduced word which contains no occurrence of $a_i^{\pm 1}$ . A set of any two subwords of $w$ of the form $u$, $u^{-1}$ is called a \emph{canceling pair in
$w$}. A family $\mathcal F$ of canceling pairs in $w$ is called
\emph{nested} if canceling pairs in $\mathcal F$ are disjoint and,
for any canceling pairs $u$, $u^{-1}$ and $v$, $v^{-1}$ in $\mathcal
F$, $v$ occurs between $u$ and $u^{-1}$ in $w$ if and only if
$v^{-1}$ does.  If $\mathcal F$ is a nested family of canceling
pairs for $w$, we abuse notation and let $w-\mathcal F$ denote the
set of subwords of $w$ which are maximal under inclusion and which
do not intersect any element of any canceling pair in $\mathcal F$
as subwords of $w$.  Finally, we define $|w-\mathcal F|_i^{simple}
:= |\mathcal F| + \sum_{w' \in (w-\mathcal F)} |w'|_i^{simple}$.
\end{definition}

\begin{lemma}\label{lem:subtosimple}
Let $w\in\F$ be a nontrivial reduced word which contains no
occurrence of $a_i^{\pm 1}$ and let $T$ be the set consisting of all
nested families of canceling pairs of $w$.  Then
    \[|w|_i^{cr}\geq\min_{\mathcal F \in T} \left(\max\left\{
    \frac{|\mathcal F|}{2}-1,~~
    \frac{1}{5}|w-\mathcal F|_i^{simple}-3
     \right\} \right).\]
\end{lemma}

\begin{proof}
Let $\gamma=v_1^{u_1}v_2^{u_2}\cdots v_l^{u_l}$ be an optimal
decomposition of $w$ realizing its conjugate reduced $i$-length.

First of all, since $\gamma$ is optimal, we may assume that all
$v_j$ are cyclically reduced (cyclic reduction of $v_j$ cannot
increase the conjugate reduced $i$-length).

Now we will utilize the technique used in the proof of the van
Kampen lemma (see, for example,~\cite{lyndon_s:comb_grp_theory}).
The word $w$ represents a trivial element in the group defined by
the presentation
    \begin{equation}
    \label{eqn:present}
    \langle \ba\ |\ v_1,v_2,\ldots,v_l\rangle.
    \end{equation}
Consider the van Kampen diagram $\Gamma_0$ with boundary label
$\gamma$ over the presentation~\eqref{eqn:present} as depicted in
Figure \ref{fig:gamma_diagram}. This diagram is a wedge of $l$
``lollipops'' corresponding to $l$ factors of $\gamma$ with
``stems'' labeled by the $u_j$ and with the ``candies'' (boundaries of
$2$-cells) labelled by the $v_j$. The base-vertex in $\Gamma_0$ is the
common vertex of ``lollipops''.

\begin{figure}[!h]
\epsfig{file=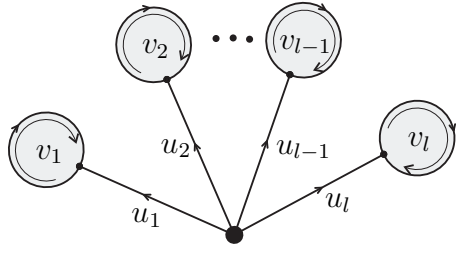} \caption{van Kampen diagram
$\Gamma_0$ corresponding to decomposition $\gamma$}
\label{fig:gamma_diagram}
\end{figure}

Fix some free reduction process transforming $\gamma$ to $w$. The
$j$th step of this reduction process takes the van Kampen diagram
$\Gamma_{j-1}$ to the diagram $\Gamma_j$, and corresponds to
modifying a pair of adjacent, inversely labeled edges along the
boundary cycle of $\Gamma_{j-1}$.  This has the effect of `removing'
this pair of edges from the boundary cycle of $\Gamma_j$ in the
following sense. If these two edges have just one vertex in common,
they are folded and if this common vertex has degree $2$ in $\Gamma_{j-1}$ then the edge obtained by folding is removed from $\Gamma_{j-1}$. If they have two vertices in common, the union of
$2$-cells bounded by these two edges is completely removed from
$\Gamma_{j-1}$. This folding or removing defines a new van Kampen
diagram $\Gamma_j$. At the end of the process we obtain a van Kampen
diagram $\Gamma$ with boundary label $w$ shown in
Figure~\ref{fig:van_kampen_bridges}. Note that in the reduction
process the number of $2$-cells in each successive van Kampen
diagram does not grow, so the number $l'$ of $2$-cells in $\Gamma$
does not exceed $l$.

\begin{figure}[!h]
\epsfig{file=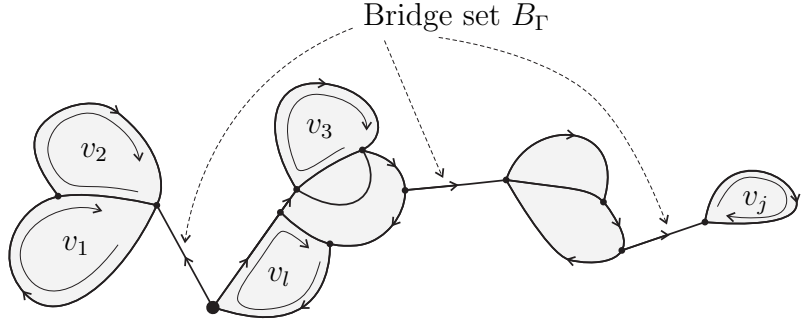} \caption{van Kampen
diagram $\Gamma$ after folding} \label{fig:van_kampen_bridges}
\end{figure}

Because each of the $v_j$'s is cyclically reduced, the boundary of
each 2-cell in the diagram $\Gamma$ is labeled by a cyclic conjugate
of $v_j$ that depends on where along the boundary one begins
reading. The \emph{bridge set} $B_\Gamma$ of $\Gamma$ is the set of
all vertices and edges whose deletion from the topological
realization $|\Gamma|$ of $\Gamma$ would disconnect it.  A
\emph{disk-component} of $\Gamma$ is a subset of $\Gamma$ which is
the closure of a connected component of $|\Gamma| - |B_\Gamma|$. The
disk-components of $\Gamma$ are joined by (possibly trivial)
edge-paths from the bridge set. Retracting each of these paths to a
point produces a new van Kampen diagram $\Gamma'$ with a boundary
label $u$ obtained from $w$ by removing a nested family of canceling
pairs, denoted $\mathcal F$, where each canceling pair corresponds to a path inside $B_\Gamma$ whose inner vertices have degree 2. Such a diagram is depicted in
Figure~\ref{fig:van_kampen_no_bridges}. Note that $u$ is not
necessarily freely reduced, but that $u$ is the product of subwords
in $w-\mathcal F$, all of which are subwords of $w$ and hence freely
reduced.  The vertices of degree at least three along the boundary
of $\Gamma'$ split $u$ into subwords $w_1,w_2,\ldots,w_k$, where
each $w_j$ is a part of the boundary of a $2$-cell in $\Gamma$. This
partition of $u$ refines the partition $w_1',w_2',\ldots,w_r'$ of
$u$ induced by $w - \mathcal F$.

Collapsing all disc components of $\Gamma$ and removing vertices of
degree 2 leaves the tree with $e$ edges and $r'$ vertices of degree
1, each of which was obtained by collapsing one of the disc
components. In every such tree we have $e\leq 2r'$. For the number
of canceling pairs in $\mathcal F$ we get $|\mathcal F|=e+r''$,
where $r''$ is the number of disc components in $\Gamma$ that
collapse to the vertices of degree $2$. But since each disc
component produces at least one $w_j'$, we get

\begin{equation}
\label{eqn:|F|}
|\mathcal F|=e+r''\leq 2r'+r''\leq2(r'+r'')\leq 2r
\end{equation}

\begin{figure}
\epsfig{file=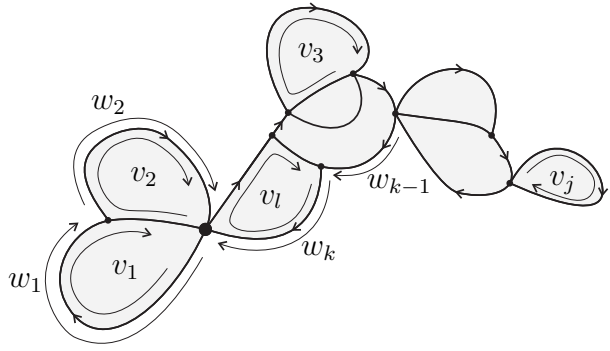} \caption{van
Kampen diagram $\Gamma'$ after bridge retraction}
\label{fig:van_kampen_no_bridges}
\end{figure}

We have by Lemma~\ref{lem:simple_length2} and the last inequality that
\begin{equation}
\label{eqn:wprimesimple} |w-\mathcal F|_i^{simple}=|\mathcal F|+
\sum_{j=1}^r|w_j'|_i^{simple}\leq 2r+
\sum_{j=1}^k|w_j|_i^{simple}+(k-r)= k + r +
\sum_{j=1}^k|w_j|_i^{simple}.
\end{equation}

By construction each $w_j$ is a subword of a cyclic conjugate $v_t'$ of some $v_t$ representing the label of the boundary of 2-cell in $\Gamma'$ to which $w_j$ belongs. It may happen that several $w_j$ lie on the boundary of one cell labelled by a conjugate of $v_t$, but by construction these occurences do not overlap. Let $\{c_1,\ldots,c_l\}$ denote the set of cells in $\Gamma'$ and assume that $v_j$ is a boundary label of $c_j$. Then the sum in~\eqref{eqn:wprimesimple} can be rewritten as

\begin{equation}
    \label{eqn:wprimesimple2}
    |w-\mathcal F|_i^{simple}
    \leq k + r + \sum_{t=1}^l\sum_{\ w_j\in\partial c_t}|w_{j}|_i^{simple}
    \end{equation}
Since by Lemmas~\ref{lem:simple_length} and~\ref{lem:simple_conj},
    \[
    \sum_{\ w_j\in\partial c_t}|w_j|_i^{simple}\leq|v_t'|_i^{simple}\leq|v_t|_i^{simple}+1,
    \]
we can transform inequality~\eqref{eqn:wprimesimple2} to
    \begin{equation}\label{eqn:wprimesimple3}
    |w-\mathcal F|_i^{simple}
    \leq k + r + \sum_{t=1}^l(|v_t|_i^{simple}+1) = |w|_i^{cr} + k + r + 1\leq |w|_i^{cr} + 2k + 1.
    \end{equation}

To finish the proof of the lemma we first prove that $k\leq
2l'-1\leq2l-1$. This follows by induction on the number $l'$ of
cells of $\Gamma$ as follows.  Clearly if $l' = 1$ then $k = 1$.
Assume that for any bridge-free van Kampen diagram with $l'-1$
2-cells the number of arcs along the boundary without vertices of
degree at least 3 is at most $2(l'-1)-1$. Choose a 2-cell $c$ in
$\Gamma$ whose boundary contains a piece $p$ of boundary of $\Gamma$
and such that after $p$ and interior of $c$ are removed from
$\Gamma$ the resulting diagram $\Gamma-p$ is still bridge-free and
connected. There are several cases describing how $p$ may be
attached to the boundary of $\Gamma-p$.  It is straightforward to
check that, in all cases, the attaching of $p$ can increase the
number of arcs without vertices of degree at least 3 by at most $2$.

Finally, consider two cases. If $|w-\mathcal F|_i^{simple} \leq 5l-5$, then
    \[|w|_i^{cr}\geq l-1\geq \frac{1}{5} |w-\mathcal F|_i^{simple}.\]
But if $|w-\mathcal F|_i^{simple} > 5l-5$, then since $k\leq2l-1<\frac{2}{5} |w-\mathcal F|_i^{simple} + 1$ we get from~\eqref{eqn:wprimesimple3}
    \[|w|_i^{cr}
    \geq |w-\mathcal F|_i^{simple} -2k -1
    >   \frac{1}{5} |w-\mathcal F|_i^{simple} -3.\]
This proves half of the lemma.

For the second part of the lemma note that by~\ref{eqn:|F|}
\[|\mathcal F|\leq 2(r'+r'')\leq 2l\]
since $r'+r''$ does not exceed the number of all disk components in
$\Gamma$ and each disk component contains at least one cell. Thus,
    $$|w|_i^{cr} \geq l - 1\geq \frac{|\mathcal F|}{2}-1.$$

The statement of the lemma now follows.
\end{proof}

\begin{corollary}\label{cor:positive}
If $w$ is a positive word, then
    $$|w|^{cr} \geq \frac{1}{5}|w|^{simple}-3.$$
\end{corollary}

\begin{proof}
If $w$ is positive, then the only possible family of canceling pairs is the trivial family.
\end{proof}

\begin{corollary}\label{cor:arbitrarylength}
There exist words of arbitrary (simple, conjugate reduced)
$i$-length; there exist bases of $\F$ of arbitrary full $i$-length.
\end{corollary}

\begin{proof}
It now follows from the previous corollary that, for $w = a_1^2a_2^2\dots a_{n-1}^2a_1$,
    $$|a_nw^l|_n = |w^l|_n^{cr} \geq l/5-3.$$
\end{proof}

We now state a lemma relating simple and conjugate reduced
$i$-length.

\begin{lemma}\label{lem:nocutvertex}
For any reduced word $w$,
    $$|w|_i^{simple} \geq |w|_i^{cr}.$$
If $|w|_i^{cr} > 0$, then the Whitehead graph $\Gamma_{\ba -
\{a_i\}}(w)$ has no cut vertex.
\end{lemma}

\begin{proof}
A word $w$ represents a decomposition of itself with one factor
whose conjugate reduced $i$-length is equal by definition to
$|w|_i^{simple}$.

If $|w|_i^{cr} > 0$, then $|w|_i^{simple} > 0$.  If the Whitehead
graph $\Gamma_{\ba - \{a_i\}}(w)$ had a cut vertex, then since $w$
is freely reduced the Whitehead graph of any subword of $w$ would
also have a cut vertex.  This contradicts that $|w|_i^{simple} > 0$.
\end{proof}

Finally, we have the two lemmas describing how conjugate reduced
$i$-length behaves under multiplication.

\begin{lemma}\label{lem:subwordmult}
For any words $u$ and $v$, we have
    $$|u|_i^{cr} - |v|_i^{cr} - 1 \leq |uv|_i^{cr} \leq |u|_i^{cr} + |v|_i^{cr} + 1.$$
\end{lemma}

\begin{proof}
A decomposition of $uv$ may be obtained by concatenating optimal
decompositions of $u$ and $v$.  The associated $i$-length of this
decomposition of $uv$ yields the second inequality.  The first
inequality follows from the second inequality by concatenating $uv$
and $v^{-1}$:  $|u|_i^{cr} = |uvv^{-1}|_i^{cr} \leq |uv|_i^{cr} +
|v|_i^{cr} + 1$.
\end{proof}

\begin{lemma}
\label{lem:uwv}
For any words $u$, $v$, and $w$, we have
    $$|uv|_i^{cr} - |w|_i^{cr} - 1 \leq |uwv|_i^{cr} \leq |uv|_i^{cr} + |w|_i^{cr} + 1.$$
\end{lemma}

\begin{proof}
We prove that $|uwv|_i^{cr} \leq |uv|_i^{cr} + |w|^{cr}_i + 1$; the
first inequality will then follow by observing $|uv|_i^{cr} =
|uww^{-1}v|_i^{cr}  \leq |uwv|_i^{cr} + |w|_i^{cr} + 1$.

Consider an optimal decomposition of $uv$,
    $$\omega = v_1^{u_1}\dots v_l^{u_l},$$
so that $|uv|_i^{cr} = \sum_j |v_j|_i^{simple}+l-1$.  We will alter
this optimal decomposition of $uv$ to obtain a decomposition of
$uwv$, at the price of possibly introducing a bounded amount of
additional conjugate reduced $i$-length.  The lemma will then
follow.

The decomposition $\omega$ freely reduces to $uv$, or put another way the word $\omega$ is obtained from $uv$ by a sequence of words, each of which differs from the previous one by inserting a single canceling pair of letters.  Throughout this process, we may split each word into two halves, the left half and the right half, as follows.  Begin by declaring $u$ is the left half of $uv$ and $v$ is the right half.  If a canceling pair is inserted into the middle of either half of a word, insert the canceling pair in the appropriate half to obtain the new halves.  If a canceling pair $b^{-1}b$ is inserted between the left half and the right half, add $b^{-1}$ to the left half and add $b$ to the right half to obtain the new halves.

Let $p$ be the smallest index such that the left half of $\omega$ is contained in $u_1^{-1}v_1u_1\dots u_p^{-1}v_pu_p$.

The split between halves of $\omega$ will either be in $u_p^{-1}$, $v_p$, $u_p$ or between them.

If the split occurs in $v_p$ then at the price of splitting
$v_p^{u_p}$ into a product $(v_p')^{u_p}(v_p'')^{u_p}$, we may
assume that the left half of $\omega$ is equal to
$u_1^{-1}v_1u_1\dots u_p^{-1}v_pu_p$, making the right half of
$\omega$ equal to $v_{p+1}^{u_{p+1}}\dots v_l^{u_l}$.  Possibly
splitting $v_p$ could increase associated conjugate reduced
$i$-length by 1.

If the split happens in, immediately before, or immediately after $u_p$ then $u_p=u_p'u_p''$ where $u_p'$ is in the left half of $\omega$ and $u_p''$ is in the right half of $\omega$. Note that $u_p'$ or $u_p''$ can be trivial.

Consider inserting $w$ into $\omega$, between $u_p'$ and $u_p''$.
Any optimal decomposition of $w$ conjugated by $(u_p'')^{-1}$
inserted into $\omega$ then yields a decomposition of $uwv$.  The
associated conjugate reduced $i$-length is
    $$|w|_i^{cr} + l + \sum_j |v_j|_i^{simple} \leq |w|_i^{cr} + |uv|_i^{cr} + 1.$$

The case when the split in $\omega$ occurs in, immediately before, or immediately after $u_p^{-1}$ is considered similarly.

This proves the upper bound, and finishes the proof.

\end{proof}

\subsection{Properties of Full $i$-Length}~\\
We now consider properties of full $i$-length -- that is, how
$i$-length behaves for a basis.

\begin{lemma}\label{lem:kis0}
For any basis $\bx$ of $\F$, any $x \in \bx$, and any subword $w$
of an $i$-chunk of $\alpha_\bx x$, we have $|w|_i^{cr} = 0$.
\end{lemma}

It follows that this result holds for any subset of any basis as well.

\begin{proof}

Throughout this proof, for sake of simplicity of notation, we write $\alpha$ for $\alpha_\bx$.  As $\bx$ is a basis, so is $\alpha\bx$. By definition, the $i$-length of an element or of a set of elements is invariant under conjugation, where we may even conjugate different elements in the set by different conjugators. Therefore cyclic reduction of all elements of $\alpha\bx$ does not change any $i$-length involved. Let $\by$ be the set $\widetilde{\alpha\bx}$ obtained from $\alpha\bx$ by cyclically reducing every element. Since $\alpha\bx$ is a basis, $\by$ is a separable set. Therefore by Theorem~\ref{thm:cutvertex} the augmented Whitehead graph $\hat\Gamma_\ba(\by)$ of $\by$ has a cut vertex. Note that this graph does not have vertex loops since each word in $\by$ is cyclically reduced.

Proof by contradiction:  assume that there exists some subword $w$
of an $i$-chunk of $\alpha x$ with $|w|_i^{cr}> 0$.  As $|w|_i^{cr}
> 0$, by Lemma \ref{lem:nocutvertex}, the subgraph $\Gamma'$ of
$\hat\Gamma_\ba(\by)$ on the vertex set corresponding to $\ba -
\{a_i\}$ has no cut vertex (since there are no vertex loops in the
graph). It remains to consider the vertices corresponding to
$a_i^{\pm 1}$ in $\hat\Gamma_\ba(\by)$.  Since $a_i$ must appear as
a letter in $\bx$ and, hence, in $\by$, by the definition of
Whitehead graph each of $a_i$, $a_i^{-1}$ has at least one neighbor
in $\hat\Gamma_\ba(\by)$.

Consider the case when either $a_i$ or $a_i^{-1}$ has exactly one neighbor in $\hat\Gamma_\ba(\by)$. Without loss of generality, assume $a_i$ has exactly one neighbor. If the neighbor $b$ of $a_i$ were in $\Gamma'$, we would contradict the definition of $w_R(\bx)$:  $b^{-1}$ should have been appended to $w_R(\bx)$.

Thus, the only neighbor of $a_i$ must be $a_i^{-1}$.  In this case, each occurrence of $a$ (resp. $a^{-1}$) in $\alpha \bx$ must be cyclically followed (resp. preceded) by $a$ (resp. $a^{-1}$). The only way for this to occur is if every element of $\by$ involving $a_i$ is some power of $a_i$.  But elements of $\by$ are primitives in $\F$ as they are conjugates of basis elements of $\alpha\bx$. Therefore this power can only be  $a_i^{\pm 1}$. Moreover, if there are two elements in $\by$ of the form $a_i^{\pm1}$, then there should be two conjugates of $a$ or $a^{-1}$ in $\alpha\bx$, which is impossible because in this case we can obtain a commutator as a primitive element of $\F$.  Thus, we may assume without loss of generality that $a_i$ is an element of $\by$ and no other element of $\by$ contains an occurrence of $a_i$.

Since $\by$ was obtained from $\alpha\bx$ by conjugating its
elements, the structure of $\alpha\bx$ is as follows. There is one
element of the form $a_i^w$ for some $w\in \F$, whose conjugate in
$\by$ is $a_i$. All other elements in $\alpha\bx$ are conjugates of
words in $\by$ not involving $a_i$ by conjugators that may generally
contain $a_i$. Then $\bz=(\alpha\bx)^{w^{-1}}$ is a basis for $\F$
one of whose elements is $a_i$ and the others are conjugates of
words in $\by$ where the words in $\by$ do not involve $a_i$ (but
the conjugators could).

By Proposition~\ref{prop:nielsen} there is a sequence $(\delta_j),1\leq j\leq t$ of elementary Nielsen transformations taking $\bz$ to the standard basis $\ba$ obtained from the Nielsen reduction process. In other words,
    \[(\bz)\Bigl(\prod_{j=1}^t\delta_j\Bigr)=\ba.\]
Since the Nielsen reduction process does not increase the length of basis elements, the element $a_i$ in $\bz$ will be invariant under each transvection $\delta_j$. Let $S=\{j:\ \delta_j\ \text{does not
involve}\ a_i\}$ and consider the basis
    \[\bu=(\ba)\Bigl(\prod_{j\in S}\delta_j\Bigr)^{-1}\]
for $\F$.  By construction this basis is obtained from $\bz$ by removing all occurrences of $a_i$ from $\bz$ except a single occurrence of $a_i$ as an element of $\bz$. This implies that all other elements of $\bu$ form a basis for $\la\ba-\{a_i\}\ra$. On the other hand, elements of the basis $\bu$ are conjugates of elements of $\bz$. Therefore cyclic reduction of elements in $\bu$ gives the set $\by$ up to cyclic conjugation. But then $\by -\{a_i\}$ is a separable set in $\la\ba - \{a_i\}\ra$, and is such that $\hat\Gamma_{\ba - \{a_i\}}(\by - \{a_i\})$ has no cut vertex. This contradicts Theorem \ref{thm:cutvertex}, and shows that neither $a_i$ nor $a_i^{-1}$ may have exactly one neighbor in $\hat\Gamma_\ba(\by)$.

We are left to consider the remaining case, when both $a_i$ and $a_i^{-1}$ have at least two neighbors in $\hat\Gamma_\ba(\by)$.  As $\Gamma'$ contains no cut vertex, the only way for $\hat\Gamma_\ba(\by)$ to still have a cut vertex in this situation is if $a_i$ and $a_i^{-1}$ both have exactly two neighbors in $\hat\Gamma_\ba(\by)$, both are neighbors of each other, and both share a common third neighbor, say $b$.  This means that every occurrence of $a_i^k$, $k \neq 0$, in $\alpha \bx$ appears by itself in $\alpha \bx$ or appears conjugated by $b^{-1}$.  But this contradicts the definition of $w_C(\bx)$:  the letter $b^{-1}$ should have been appended to $w_C$.

\end{proof}

As corollaries of the above lemma and definition of the $i$-length
of a set we get the following statements.

\begin{corollary}
\label{cor:noai}
For any basis $\bx$ of $\F$ and any $x \in \bx$,  $k(\alpha_x x) \leq |\alpha_{\bx} x|_i = 0$, and so
    $$|\bx|_i = |\wR{\bx}\wL{\bx}|^{cr}_i.$$
\end{corollary}

\begin{lemma}\label{lem:Xtox}
For any basis $\bx$ and any $x \in \bx$ containing $a_i$,
    $$|\bx|_i - 2 \leq |x|_i \leq |\bx|_i + 2.$$
\end{lemma}

\begin{proof}

By Lemma \ref{lem:kis0}, for any $x \in \bx$,  $|\alpha_\bx
x|_i^{cr} = 0$.  Without loss of generality, as $i$-length is
unaffected by conjugation assume that $x$ is such that all of $x$,
$\alpha'_\bx x$, and $\alpha_\bx x$ are cyclically reduced.  For
simplicity of notation, let $\alpha' := \alpha'_\bx$.  Note that
every occurrence of $a_i$ in $\alpha' x$ occurs in a subword of
$\alpha' x$ in at least one of the following four forms: $a_i$,
$\wC{\bx}^{-1}a_i$, $a_i\wC{\bx}$, $\wC{\bx}^{-1}a_i\wC{\bx}$.
Similarly, every occurrence of $a_i^{-1}$ in $\alpha'x$ occurs in a
subword of $\alpha' x$ in at least one of the forms:  $a_i^{-1}$,
$\wC{\bx}^{-1}a_i^{-1}$, $a_i^{-1}\wC{\bx}$,
$\wC{\bx}^{-1}a_i^{-1}\wC{\bx}$.

Consider the following possible cases.

\begin{enumerate}
\item[Case 1.] The word $\alpha' x$ is a power of $a_i$, so that no occurrence of $a_i$ (or its inverse) appears in $x$ multiplied by $\wC{\bx}$ (or its inverse).  In this case $\wL{x}=\wR{\bx}\wL{\bx}$ and $\wR{x}$ is trivial, therefore $\wR{x}\wL{x}=\wR{\bx}\wL{\bx}$ and $|x|_i=|\bx|_i$ by Corollary~\ref{cor:noai}.

\item[Case 2.] Some occurrences of $a_i$ (or its inverse) in $\alpha' x$ occur in subwords of $\alpha' x$ of the form $\wC{\bx}^{-1}a_i\wC{\bx}$ (resp. $\wC{\bx}^{-1}a_i^{-1}\wC{\bx}$), while some do not. Note that the last letter in $\wC{\bx}$ must differ from the last letter in $\wR{\bx}$ since these letters do not cancel in $\wR{\bx}\wC{\bx}^{-1}$. But this implies that $\wL{x}=\wL{\bx}$ and $\wR{x}=\wR{\bx}$, so $k(x) = k(\bx)$, again yielding $|x|_i=|\bx|_i$.

\item[Case 3.] Every $a_i$ (resp. $a_i^{-1}$) in $x$ occurs in $\alpha' x$ in a subword of $\alpha' x$ of the form $\wC{\bx}^{-1}a_i\wC{\bx}$ (resp. $\wC{\bx}^{-1}a_i^{-1}\wC{\bx}$). Then $\wL{x}$ contains $\wC{\bx}^{-1}\wL{\bx}$ as a terminal segment.  It may also contain some portion $w_2$ of an $i$-chunk of $\alpha_\bx(x)$ of zero $i$-length by Lemma~\ref{lem:kis0}, and finally it may contain some portion of $\wR{\bx}\wC{\bx}$. In any case $\wR{x}$ will contain the rest of $\wR{\bx}\wC{\bx}$ and possibly some portion $w_1$ of an $i$-chunk of $\alpha_\bx(x)$ also of zero $i$-length. Therefore
    \[\wR{x}\wL{x}=\wR{\bx}\wC{\bx}w_1w_2\wC{\bx}^{-1}\wL{\bx},\]
where $w_1$ and $w_2$ may be trivial.  It follows that $k(x) \leq
k(\bx) = 0$, so $|x|_i = |\wR{x}\wL{x}|_i^{cr}$.  Therefore applying
Lemmas~\ref{lem:uwv} and \ref{lem:subwordmult}, and taking into
account that conjugation does not change the conjugate reduced
$i$-length of a subword, we obtain
\begin{eqnarray*}
    |x|_i=|\wR{x}\wL{x}|^{cr}_i
        &=&     |\wR{\bx}\wC{\bx}^{-1}w_1w_2\wC{\bx}\wL{\bx}|^{cr}_i\\
        &\leq&  |\wR{\bx}\wL{\bx}|^{cr}_i+|\wC{\bx}^{-1}w_1w_2\wC{\bx}|^{cr}_i+1\\
        &=&     |\wR{\bx}\wL{\bx}|^{cr}_i+|w_1w_2|^{cr}_i+1\\
        &\leq&  |\wR{\bx}\wL{\bx}|^{cr}_i+|w_1|^{cr}_i+|w_2|^{cr}_i+2\\
        &=&     |\wR{\bx}\wL{\bx}|^{cr}_i+2=|\bx|_i+2.
\end{eqnarray*}

Similarly we derive the first inequality
    \begin{eqnarray*}
    |\bx|_i=|\wR{\bx}\wL{\bx}|^{cr}_i
        &=&     |\wR{\bx}\wC{\bx}^{-1}w_1w_2(w_1w_2)^{-1}\wC{\bx}\wL{\bx}|^{cr}_i\\
        &\leq&  |\wR{\bx}\wC{\bx}^{-1}w_1w_2\wC{\bx}\wL{\bx}|^{cr}_i+|w_1w_2|^{cr}_i+1\\
        &\leq&  |\wR{x}\wL{x}|^{cr}_i+|w_1|^{cr}_i+|w_2|^{cr}_i+2\\
        &=&     |\wR{x}\wL{x}|^{cr}_i+2=|x|_i+2.
    \end{eqnarray*}

Note that the last case covers all situations when $\wC{\bx}$ is trivial.
\end{enumerate}
\end{proof}

As an immediate corollary we obtain the following statement.

\begin{corollary}
\label{cor:bxby}
For any bases $\bx$ and $\by$ sharing a common element containing $a_i$,
    $$\bigl|\ |\bx|_i - |\by|_i\bigr| \leq 4.$$
\end{corollary}

\begin{proof}
Let $x$ be a common element of $\bx$ and $\by$ containing $a_i$.
Then by Lemma~\ref{lem:Xtox} we get
    $$\bigl|\ |x|_i - |\bx|_i\bigr| \leq 2,$$
    $$\bigl|\ |x|_i - |\by|_i\bigr| \leq 2.$$
Combining the above inequalities proves the lemma.
\end{proof}

The following corollary is not used in the rest of the paper, but is
an interesting observation on its own.

\begin{corollary}
For any basis $\bx$ of $\F$ there is at most one
$i\in\{1,2,\ldots,n\}$ such that $|\bx|_i>0$.
\end{corollary}

\begin{proof}
Suppose that $|\bx|_i>0$ for some $i$. Let $\tilde\bx$ be the
separable set obtained from $\bx$ by cyclically reducing all of its
elements. Consider the augmented Whitehead graph $\hat\Gamma$ of
$\tilde\bx$. By construction, $\hat\Gamma$ includes the Whitehead
graphs of all $i$-chunks of all elements of $\tilde\bx$ as
subgraphs. Since by definition of full $i$-length,
$|\tilde\bx|_i=|\bx|_i>0$ we must have at least one $i$-chunk $w$ of
$\tilde\bx$ with $|w|_i^{cr}>0$, which implies by
Lemma~\ref{lem:nocutvertex} that corresponding Whitehead graph
$\Gamma_i=\Gamma_{\ba-a_i}(w)$ does not have a cut vertex. But
according to Theorem~\ref{thm:cutvertex} graph $\hat\Gamma$ has a
cut vertex.

As $\Gamma_i$ is a subgraph of $\hat\Gamma$, it must be that
removing any cut vertex $v$ of $\hat\Gamma$ creates at least two
connected components, with $\Gamma_i - \{v\}$ in one component and
at least one of $a_i$ or $a_i^{-1}$ in a different component.
Without loss of generality, say $a_i$ appears in a different
component.  Then $a_i$ can have at most one neighbor in $\Gamma_i
\subset \hat\Gamma$.  If $a_i$ has no such neighbor, then
either $\{a_i\}$ or $\{a_i,a_i^{-1}\}$ will be a connected component in $\hat\Gamma$.
Therefore either $\{a_i\}$ or $\{a_i,a_i^{-1}\}$ will be a
nontrivial proper connected component in  $\Gamma_j$ for any $j\neq
i$. But this means that any vertex in $\Gamma_j$, except possibly
$a_i$ or $a_i^{-1}$, will be a cut vertex.  By Lemma~\ref{lem:nocutvertex} we have $|\bx|_j=0$.

If $a_i$ has a neighbor in $\Gamma_i \subset \hat\Gamma$ -- without
loss of generality, say $a_k$ for some $k \neq i$ -- and $a_i$ is
not adjacent to $a_i^{-1}$, then $a_i$ is isolated in $\Gamma_k$,
and $a_k$ is a cut vertex in $\Gamma_j$ for $j \neq i, k$.  In this case, $|\bx|_j = 0$ for $j \neq i$.

If $a_i$ has a neighbor in $\Gamma_i \subset\hat \Gamma$ -- again, say $a_k$ for $k \neq i$ -- and $a_i$ is adjacent to $a_i^{-1}$, then
$a_i^{-1}$ may have no other neighbors in $\Gamma_i \subset
\hat\Gamma$ except $a_k$ because otherwise $\hat\Gamma$ would not
have a cut vertex. Here again, either $\{a_i\}$ or
$\{a_i,a_i^{-1}\}$ will be a nontrivial proper connected component
in  $\Gamma_k$, so $|\bx|_k=0$. Finally, for every $j\neq i,k$ the
vertex $a_k$ will be a cut vertex in $\Gamma_j$ yielding
$|\bx|_j=0$.
\end{proof}

\section{The Geometry of $\ES$}\label{sec:geometry}

\subsection{Distance and $i$-Length}~\\
We are now ready to estimate distances in $\ES$, based on how much $i$-length can change in a single proper nonempty index set.

\begin{lemma}
\label{lem:BasisLength} For any proper nonempty subset $S$ of the
index set $\{1, \dots, n\}$, any basis $\bx = \{x_1, \dots, x_n\}$
of $\F$, and any $S$-transformation $\phi \in \Aut$ which is the
identity on $\bx_{\ol S}$, we have that
    $$|\bx|_i - 12 \leq |\bx\phi|_i \leq |\bx|_i + 12.$$
\end{lemma}

\begin{proof}
If there exists some $x \in \bx_{\ol S}$ such that $x$ contains an
occurrence of $a_i^{\pm 1}$, then since $x \in \bx \cap \bx\phi $,
by Corollary~\ref{cor:bxby}
    $$\bigl|\ |\bx|_i - |\bx\phi|_i\bigr| \leq 4,$$
and the lemma follows.

If no such $x$ exists, choose any $x \in \bx_{\ol S}$ and let $y \in
\bx_{S}$ be an element of $\bx$ which contains an occurrence of
$a_i^{\pm 1}$. Let $\bx' := (\bx -\{x\})\cup \{xy\}$.

The bases $\bx$ and $\bx'$ share $y$ in common, so by
Corollary~\ref{cor:bxby}
\begin{equation}
\label{eqn:bx_bxp}
\bigl|\ |\bx|_i - |\bx'|_i\bigr| \leq 4,
\end{equation}

Since $\bx'$ and $\bx'\phi$ share $xy$ in common, we get
\begin{equation}
\label{eqn:bxp_phi_bxp} \bigl|\ |\bx'|_i - |\bx'\phi|_i\bigr| \leq
4,
\end{equation}

Also there must be an element among $z\in (\bx\phi)_S$ containing an
occurrence of $a_i$ (otherwise $\bx\phi$ would not contain an
occurrence of $a_i$). This element $z$ is common for bases $\bx\phi$
and $\bx'\phi$ yielding
\begin{equation}
\label{eqn:phi_bx_phi_bxp} \bigl|\ |\bx'\phi|_i - |\bx\phi|_i\bigr|
\leq 4,
\end{equation}

Finally, combining the inequalities~\eqref{eqn:bx_bxp}, \eqref{eqn:bxp_phi_bxp} and~\eqref{eqn:phi_bx_phi_bxp}, we obtain the statement of the lemma.

\end{proof}

It seems that, with more careful bookkeeping, the constant 12 might be able to be improved.

\begin{corollary}\label{cor:qi}
Let $\bx$ be a basis of $\F$.  Then the number of index changes
required in a transformation from $\Aut$ taking $\ba$ to $\bx$ is
bounded below by $\frac{1}{24}|\bx|_i - 1$.
\end{corollary}

\begin{proof}
For any index set $S$, an $S$-transformation can be written as a
product of an $S$-transformation which is identity on $\bx_S$ and an
$S$-transformation which is identity on $\bx_{\ol S}$. Applying
Lemma \ref{lem:BasisLength} twice, we see an $S$-transformation can
change $i$-length by at most 24.  The corollary then follows from
Theorem \ref{thm:geodesics}. Note that the requirement about the
compatibility of the neighboring index sets cannot decrease the number
of subwords in the optimal decomposition realizing the minimal number
of index changes.
\end{proof}

This corollary, combined with Theorem \ref{thm:geodesics}, shows our main computational theorem:

\begin{theorem}\label{thm:lowerbound}
Let $\bx$ be a basis of $\F$, expressed in terms of a fixed standard basis $\ba$.  For any index $i$ and any index sets $S_a$ and $S_x$,
    $$d_{\ES}((\ba, S_a), (\bx, S_x)) \geq \frac{|\bx|_i}{24}-1.$$
\end{theorem}

\subsection{$\ES$ is Not Hyperbolic}~\\
Corollary \ref{cor:qi} is useful for estimating distances in
$\ES$.  For instance, we may now apply this corollary to show
that $\ES$ is not hyperbolic in the sense of Gromov, by
identifying quasiflats -- that is, a quasiisometric embedding $\R^k
\to \ES$ for $k > 1$.

Let $p_t := a_1^{t+1}a_2^{t+1}\cdots
a_{n-1}^{t+1}a_1^{t+1}a_2^{t+1}a_1^{t+1}$. Note that for $t\geq 1$
the augmented Whitehead graph of $p_t$ looks similar to the graph
shown in Figure~\ref{fig:whitehead}, and removing vertices
corresponding to $a_n$ and $a_n^{-1}$ will produce graphs without
cut vertices. We propose to map the integer lattice $\Z^m$
quasiisometrically into $\ES$ by the map $\psi$ which takes
$(k_1,k_2,\ldots,k_m) \in \Z^m$ to the vertex
$\psi(k_1,k_2,\ldots,k_m)=(\bx, S)$ of $\ES$, where $S$ is
an arbitrary proper nontrivial subset of $\{1,2,\ldots,n\}$ and
$\bx$ is obtained from the standard basis $\ba$ by replacing
$a_n$ by $a_np_1^{k_1}p_2^{k_2}\cdots p_m^{k_m}$.

\begin{theorem}\label{thm:quasiflat}
The map $\psi$ yields an $m$-dimensional quasiflat in $\ES$.
\end{theorem}

\begin{proof}
To see that $\psi$ is indeed a quasiisometry, consider the images of
two points, $(k_1,k_2,\ldots,k_m)$ and $(l_1,l_2,\ldots,l_m)$ under
$\psi$.  In the domain, these points are of distance
\[d=\sum_{t=1}^m|k_t-l_t|\]
apart. In the codomain, the distance between
$\psi(k_1,k_2,\ldots,k_m)$ and $\psi(l_1,l_2,\ldots,l_m)$ is the
same as the distance between the basepoint $\ba$ and the point
represented by the standard basis with $a_n$ replaced by
$a_n\omega$, where
\[\omega = p_m^{-k_m}p_{m-1}^{-k_{m-1}}\cdots p_1^{-k_1}
p_1^{l_1}p_2^{l_2}\cdots p_m^{l_m}\] after free reduction.

By Theorem~\ref{thm:lowerbound} and the definition of full
$i$-length, the latter distance is bounded below by
    \begin{equation}\label{eqn:bxtoomega}
    \frac{1}{24}|\omega|_n^{cr}-1.
    \end{equation}

We claim $|\omega|_n^{cr} \geq \frac{d}{11}-\frac{21}{11}$, as
follows.  By Lemma~\ref{lem:subtosimple} $|\omega|_n^{cr}$ can be
estimated from below by
\begin{equation}
\label{eqn:omegaprime}
    |\omega|_n^{cr}\geq\min_{\mathcal F \in S} \left(\max\left\{
    \frac{|\mathcal F|}{2}-1,~~
    \frac{1}{5}|\omega-\mathcal F|_n^{simple}-3
     \right\} \right),
\end{equation}
where $S$ is the set consisting of all nested families of canceling
pairs in $\omega$. Let $\mathcal F$ denote the family of canceling
pairs in $\omega$ that minimizes the bound
in~\eqref{eqn:omegaprime}.

There may be free cancellations of two types in $\omega$. First,
several full occurrences of $p_t$ may cancel with full occurrences
of $p_t^{-1}$ in the middle where $p_1^{-k_1}$ and $p_1^{l_1}$ meet,
and second, there may be cancellation of two occurrences of $p_t$
for different $t$. In the second case, by the definition of $p_t$,
the only part that may cancel is a subword of either the last and/or
the first syllable of the form $a_1^{t+1}$.  However, reductions of
the second type preserve Whitehead graphs in the following sense:
the Whitehead graph of the uncanceled subword $q$ of every copy of
$p_t^{\pm1}$ (with vertices $a_n^{\pm1}$ removed) will still have a
cut vertex. We will call such a subword $q$ a \emph{leftover of type
$t$} and denote by $q^{(t)}$. Each $q^{(t)}$ contains
$a_2^{t+1}\dots a_{n-1}^{t+1}a_1^{t+1}a_2^{t+1}$ as a subword.

Consider canceling pairs in $\mathcal F$. Every occurrence of
$q^{(t)}$ disjoint from pairs in $\mathcal F$ will introduce 1 to
the sum in the definition of $|\omega-\mathcal F|^{simple}_n$.
Therefore we just need to count the number of such $q^{(t)}$'s
to get a lower bound for $|\omega-\mathcal F|^{simple}_n$.  A canceling pair in $\mathcal F$ may contain some $q^{(t)}$ in one factor and $(q^{(t)})^{-1}$ in the other factor.  We remove any such $q^{(t)}$ from our count by throwing out all but $|k_t-l_t|$ occurrences of any $q^{(t)}$ for each $t$.  This leaves exactly $d$ possible $q^{(t)}$'s to count.  For the remaining $q^{(t)}$'s, by the definition of $p_t$, a given canceling pair may involve at most 4 different occurrences of a $q^{(t)}$ (for possibly different values of $t$).  Thus, removing all $p_t^{\pm1}$ that cancel in the free
reduction of the first type in the previous paragraph, then all
leftovers that are contained in a canceling pair which also contains the inverse of the leftover, and finally
removing all leftovers that intersect a canceling pair at all will leave at
least $d-4|\mathcal F|$ occurrences of a $q^{(t)}$.

Hence,
    \[\frac{1}{5}|\omega-\mathcal F|^{simple}_n-3
    \geq\frac{1}{5}\bigl((d-4|\mathcal F|)+|\mathcal F|\bigr)-3
    =\frac{d}{5}-\frac{3}{5}|\mathcal F|-3.\]
From~\eqref{eqn:omegaprime} we obtain
\[
    |\omega|_n^{cr}\geq\max\left\{
    \frac{|\mathcal F|}{2}-1,~~
    \frac{d}{5}-\frac{3}{5}|\mathcal F|-3
     \right\},
\]
If $|\mathcal F|\geq \frac{2}{11}d-\frac{20}{11}$ then
    \[|\omega|^{cr}_n\geq \frac{|\mathcal F|}{2}-1\geq \frac{d}{11}-\frac{21}{11}.\]
But if $|\mathcal F|<\frac{2}{11}d-\frac{20}{11}$ then
    \[|\omega|^{cr}_n\geq \frac{d}{5}-\frac{3}{5}|\mathcal F|-3> \frac{d}{11}-\frac{21}{11}.\]

In either case, as claimed, $|\omega|_n^{cr} \geq
\frac{d}{11}-\frac{21}{11}$.

Combining this claim with the lower bound in \eqref{eqn:bxtoomega}, we have that the
distance between the vertices $\psi(k_1,k_2,\ldots,k_m)$ and
$\psi(l_1,l_2,\ldots,l_m)$ is bounded below by $\frac{d}{24\cdot
11}-\frac{21}{24\cdot 11}-1=\frac{1}{264}d-\frac{95}{88}$.

We claim the distance between $\psi(k_1,k_2,\ldots,k_m)$ and $\psi(l_1,l_2,\ldots,l_m)$ is also bounded above by $d+m$, as follows.  Without loss of generality, fix $S = \{n\}$.  Recall a vertex in $\ES$ is defined up to conjugation.  Thus, for any word $w$,
    $$[\la a_1, \dots, a_{n-1}\ra*\la a_n\ra] = [\la a_1^w, \dots, a_{n-1}^w \ra*\la a_n^w\ra].$$
Thus, the following describes a path from $\psi(k_1,k_2,\ldots,k_m)$ to $\psi(l_1,l_2,\ldots,l_m)$:
\begin{eqnarray*}
\psi(k_1, k_2, \dots, k_m)
  &=&   [\la a_1, \dots, a_{n-1} \ra * \la a_np_1^{k_1}p_2^{k_2}\cdots p_m^{k_m}\ra ]\\
  &=&   [\la a_1, \dots, a_{n-1} \ra * \la p_1^{-k_1}\dots p_m^{-k_m}a_n\ra ]\\
  &\to& [\la a_1, \dots, a_{n-1} \ra * \la p_1^{-k_1}\dots p_m^{-k_m}a_np_1^{l_1-k_1}\ra ]\\
  &=&   [\la a_1, \dots, a_{n-1} \ra * \la p_2^{-k_2}\dots p_m^{-k_m}a_np_1^{l_1}\ra ]\\
  &\to& [\la a_1, \dots, a_{n-1} \ra * \la p_2^{-k_2}\dots p_m^{-k_m}a_np_1^{l_1}p_2^{l_2-k_2}\ra ]\\
  &=&   [\la a_1, \dots, a_{n-1} \ra * \la p_3^{-k_3}\dots p_m^{-k_m}a_np_1^{l_1}p_2^{l_2}\ra ]\\
  &\to& \dots \\
  &\to& [\la a_1, \dots, a_{n-1} \ra * \la p_m^{-k_m}a_np_1^{l_1}p_2^{l_2}\dots p_{m-1}^{l_{m-1}}p_m^{l_m-k_m}\ra ]\\
  &=&   [\la a_1, \dots, a_{n-1} \ra * \la a_np_1^{l_1}p_2^{l_2}\dots p_{m-1}^{l_{m-1}}p_m^{l_m}\ra ]\\
  &=&   \psi(l_1, l_2, \dots, l_m).
\end{eqnarray*}
At each arrow, the above sequence is the same:  for some integers $i$ and $j$, we append $p_i^j$ to the last factor in a vertex of the form $[\la a_1, \dots, a_{n-1} \ra * \la u a v\ra]$.  We claim this may be done in via at most $2j+1$ steps in $\ES$, by the following edge path, where in this sequence, arrows each represent crossing exactly 1 edge of $\ES$:
\begin{eqnarray*}
  &&        [\la a_1, \dots, a_{n-1} \ra * \la u a v \ra] \\
  &\to& [\la a_{n-1} \ra * \la a_1, \dots, a_{n-2},  u a_n v \ra] \\
  &=&   [\la a_{n-1} \ra * \la a_1, \dots, a_{n-2},  u a_n v a_1^{i+1}a_2^{i+1}\cdots a_{n-2}^{i+1} \ra] \\
  &\to& [\la a_3, \dots, a_{n-2} \ra * \la a_1,a_2, a_{n-1},  u a_n v a_1^{i+1}a_2^{i+1}\cdots a_{n-2}^{i+1} \ra] \\
  &=&   [\la a_3, \dots, a_{n-2} \ra * \la a_1,a_2, a_{n-1},  u a_n v p_i \ra] \\
  &\to& [\la a_{n-1} \ra * \la a_1, \dots, a_{n-2},  u a_n v p_i \ra] \\
  &=&   [\la a_{n-1} \ra * \la a_1, \dots, a_{n-2},  u a_n v p_i a_1^{i+1}a_2^{i+1}\cdots a_{n-2}^{i+1} \ra] \\
  &\to& [\la a_3, \dots, a_{n-2} \ra * \la a_1,a_2, a_{n-1},  u a_n v p_i a_1^{i+1}a_2^{i+1}\cdots a_{n-2}^{i+1} \ra] \\
  &=&   [\la a_3, \dots, a_{n-2} \ra * \la a_1,a_2, a_{n-1},  u a_n v p_i^2 \ra] \\
  &\to& \dots \\
  &=&   [\la a_3, \dots, a_{n-2} \ra * \la a_1,a_2, a_{n-1},  u a_n v p_i^j \ra] \\
  &\to& [\la a_1, \dots, a_{n-1} \ra * \la u a v p_i^j \ra]
\end{eqnarray*}
Note here we shown the path when $j > 0$; the path when $j < 0$ is similar.  Combining these two descriptions, we have that:
    $$d_{\ES}(\psi(k_1,k_2, \dots, k_m), \psi(l_1, l_2, \dots, l_m)) \leq \sum_{i = 1}^{m} 2|l_i - k_i| + 1 = 2d+m.$$

As distances are bounded both above and below, we have a quasiisometry.
\end{proof}

As immediate corollaries, we obtain:

\begin{corollary}
\label{cor:nothyperbolic}
The graph $\ES$ is not hyperbolic in the sense of Gromov.
\end{corollary}

This shows that $\ES$ does not have the hyperbolicity desired for an analogue for $\Out$ of the curve complex for the mapping class group.  The hyperbolicity of the curve complex was shown by Masur and Minsky \cite{MasurMinsky}, and has proven to be useful in numerous situations.

\begin{corollary}
\label{cor:infinitedimension}
The space $\ES$ has infinite asymptotic dimension.  The dimension of every asymptotic cone of $\ES$ is infinite.
\end{corollary}

\begin{corollary}
\label{cor:notqi} The identity map on vertices between $\ES$ and
$\FFZ$ is not a quasiisometry. Moreover, there is no
coarsely $\Out$-equivariant quasiisometry between $\ES$ and $\FFZ$.
\end{corollary}

\begin{proof}
The first half follows immediately from Theorem \ref{thm:quasiflat},
as the set $\psi\Z^m \subset \ES$ has diameter 1 in $\FFZ$:
for $k \neq n$, the element $a_k$ has translation length 0 on the
Bass-Serre tree of every element in $\psi\Z^m$.

For the second half we note that since $\Out$ acts on both
$\ES$ (and $\FFZ$) by isometries, for each $\phi\in \Out$
the orbits under powers of $\phi$ of vertices of $\ES$ (and
$\FFZ$) are either all bounded or all unbounded. Now consider
$\phi\in \Out$ taking the standard basis $\ba$ of $\F$ to the
basis obtained from $\ba$ by replacing $a_n$ with $a_np_1$. By
construction, for each index set $S$ the orbit of $(\ba, S)$ in
$\ES$ under iterations of $\phi$ is unbounded as the $i$-length
of $\phi^n(\ba)$ grows. But the same orbit is bounded in $\FFZ$.
Thus, there is no coarsely $\Out$-equivariant quasiisometry between
$\ES$ and $\FFZ$ because every orbit in $\FFZ$ under
iterations of $\phi$ is bounded and cannot be an image of an
unbounded orbit in $\ES$.
\end{proof}

An analogous results with identical proofs hold true for the
relationships between the free factorization graph $\ES$ and the free factor graph $\FF$ and between $\ES$ and the free splitting graph $\FS$. There is a natural (coarsely
well-defined for $n>2$) map $\Sigma\colon \ES\to\FF$
defined by sending a vertex $[A * B]$ in $\ES$ to the vertex
$[A]$ in $\FF$.  This map is induced by the same map on
vertices from $\FFZ$ to $\FF$, which is a coarsely
$\Out$-equivariant quasiisometry. Also there is a natural embedding $\imath\colon \ES\to\FS$ defined by sending a vertex $[A*B]$ in $\ES$ to the vertex $[A*B]$ in $\FS$, which is quasisurjection. However, neither of the above maps is a quasiisometry.

\begin{corollary}
\label{cor:notqi2} The maps $\Sigma\colon \ES\to\FF$ and $\imath\colon \ES\to\FS$ are not quasiisometries. Moreover, there is no coarsely $\Out$-equivariant quasiisometry between $\ES$ and $\FF$, and between $\ES$ and $\FS$.
\end{corollary}

The last corollary provides a negative answer to a question of Bestvina and Feighn (the first half of Question 4.4
in~\cite{BestvinaFeighn}).

\bibliographystyle{alpha}
\bibliography{outfn2}

\def\cprime{$'$}
\begin{thebibliography}{BKMM10}

\bibitem[AS09]{AramayonaSouto}
Javier Aramayona and Juan Souto.
\newblock Automorphisms of the graph of free splittings.
\newblock Preprint: arXiv:0909.3660, 2009.

\bibitem[BBC09]{BehrstockBestvinaClay}
Jason Behrstock, Mladen Bestvina, and Matt Clay.
\newblock Growth of intersection numbers for free group automorphisms.
\newblock Preprint: arxiv:0806.4975, 2009.

\bibitem[BBK10]{BestvinaBroombergFujiwara:asdimMCG}
Mladen Bestvina, Kenneth Bromberg, and Fujiwara Koji.
\newblock The asymptotic dimension of map- ping class groups is finite.
\newblock Preprint: arXiv:1006.1939, 2010.

\bibitem[BF10]{BestvinaFeighn}
Mladen Bestvina and Mark Feighn.
\newblock A hyperbolic {${\rm Out}(F_n)$}-complex.
\newblock {\em Groups Geom. Dyn.}, 4(1):31--58, 2010.

\bibitem[BK10]{BerchenkoKogan}
Yakov Berchenko-Kogan.
\newblock Distance in the ellipticity graph.
\newblock Preprint: arXiv:1006.4853, 2010.

\bibitem[BKMM10]{BehrstockKleinerMinskyMosher}
Jason Behrstock, Bruce Kleiner, Yair Minsky, and Lee Mosher.
\newblock Geometry and rigidity of mapping class groups.
\newblock Preprint: arXiv:0801.2006, 2010.

\bibitem[Bon91]{Bonahon}
Francis Bonahon.
\newblock Geodesic currents on negatively curved groups.
\newblock In {\em Arboreal group theory ({B}erkeley, {CA}, 1988)}, volume~19 of
  {\em Math. Sci. Res. Inst. Publ.}, pages 143--168. Springer, New York, 1991.

\bibitem[CV86]{culler_v:outer_space}
Marc Culler and Karen Vogtmann.
\newblock Moduli of graphs and automorphisms of free groups.
\newblock {\em Invent. Math.}, 84(1):91--119, 1986.

\bibitem[DP10]{day_p:partial_bases}
Matthew Day and Andrew Putman.
\newblock The complex of partial bases for $f_n$ and finite generation of the
  torelli subgroup of $aut(f_n)$.
\newblock Preprint: arXiv:1006.4853, 2010.

\bibitem[DS10]{DranishnikovSapir}
Alexander Dranishnikov and Mark Sapir.
\newblock On the dimension growth of groups.
\newblock Preprint: arXiv:1008.3868, September 2010.

\bibitem[Far06]{Farley}
Daniel Farley.
\newblock Homology of tree braid groups.
\newblock In {\em Topological and asymptotic aspects of group theory}, volume
  394 of {\em Contemp. Math.}, pages 101--112. Amer. Math. Soc., Providence,
  RI, 2006.

\bibitem[Gui05]{Guirardel}
Vincent Guirardel.
\newblock C\oe ur et nombre d'intersection pour les actions de groupes sur les
  arbres.
\newblock {\em Ann. Sci. \'Ecole Norm. Sup. (4)}, 38(6):847--888, 2005.

\bibitem[Hat95]{Hatcher:HomologicalStability}
Allen Hatcher.
\newblock Homological stability for automorphism groups of free groups.
\newblock {\em Comment. Math. Helv.}, 70(1):39--62, 1995.

\bibitem[HV98a]{HatcherVogtmann:CerfTheory}
Allen Hatcher and Karen Vogtmann.
\newblock Cerf theory for graphs.
\newblock {\em J. London Math. Soc. (2)}, 58(3):633--655, 1998.

\bibitem[HV98b]{HatcherVogtmann:Complex_of_free_factors}
Allen Hatcher and Karen Vogtmann.
\newblock The complex of free factors of a free group.
\newblock {\em Quart. J. Math. Oxford Ser. (2)}, 49(196):459--468, 1998.

\bibitem[Ji04]{Ji:AsDimArithmetic}
Lizhen Ji.
\newblock Asymptotic dimension and the integral {$K$}-theoretic {N}ovikov
  conjecture for arithmetic groups.
\newblock {\em J. Differential Geom.}, 68(3):535--544, 2004.

\bibitem[Kap06]{Kapovich:Currents}
Ilya Kapovich.
\newblock Currents on free groups.
\newblock In {\em Topological and asymptotic aspects of group theory}, volume
  394 of {\em Contemp. Math.}, pages 149--176. Amer. Math. Soc., Providence,
  RI, 2006.

\bibitem[KL09]{kapovich_l:analogues_of_curve_complex}
Ilya Kapovich and Martin Lustig.
\newblock Geometric intersection number and analogues of the curve complex for
  free groups.
\newblock {\em Geom. Topol.}, 13(3):1805--1833, 2009.

\bibitem[LS01]{lyndon_s:comb_grp_theory}
Roger~C. Lyndon and Paul~E. Schupp.
\newblock {\em Combinatorial group theory}.
\newblock Classics in Mathematics. Springer-Verlag, Berlin, 2001.
\newblock Reprint of the 1977 edition.

\bibitem[Lus04]{Lustig:IntersectionForm}
Martin Lustig.
\newblock A generalized intersection form for free groups.
\newblock Preprint, 2004.

\bibitem[MKS04]{MagnusKarassSolitar}
Wilhelm Magnus, Abraham Karrass, and Donald Solitar.
\newblock {\em Combinatorial group theory}.
\newblock Dover Publications Inc., Mineola, NY, second edition, 2004.
\newblock Presentations of groups in terms of generators and relations.

\bibitem[MM99]{MasurMinsky}
Howard~A. Masur and Yair~N. Minsky.
\newblock Geometry of the complex of curves. {I}. {H}yperbolicity.
\newblock {\em Invent. Math.}, 138(1):103--149, 1999.

\bibitem[Nie24]{nielsen:aut_fn}
Jakob Nielsen.
\newblock Die {I}somorphismengruppe der freien {G}ruppen.
\newblock {\em Math. Ann.}, 91(3-4):169--209, 1924.

\bibitem[Sch06]{Schleimer:CurveComplex}
Saul Schleimer.
\newblock Notes on the complex of curves.
\newblock Preprint: http://www.warwick.ac.uk/\~{}masgar/Maths/notes.pdf, 2006.

\bibitem[Sta99]{stallings:handlebodies}
John~R. Stallings.
\newblock Whitehead graphs on handlebodies.
\newblock In {\em Geometric group theory down under ({C}anberra, 1996)}, pages
  317--330. de Gruyter, Berlin, 1999.

\bibitem[Vog02]{vogtmann:aut_fn_and_outer_space}
Karen Vogtmann.
\newblock Automorphisms of free groups and outer space.
\newblock In {\em Proceedings of the {C}onference on {G}eometric and
  {C}ombinatorial {G}roup {T}heory, {P}art {I} ({H}aifa, 2000)}, volume~94,
  pages 1--31, 2002.

\bibitem[Whi36]{whitehead:graph}
J.~H.~C. Whitehead.
\newblock On certain sets of elements in a free group.
\newblock {\em Proc. London Math. Soc.}, 41:48--56, 1936.

\end{thebibliography}

\end{document}